\documentclass[11pt]{amsart}

\newtheorem{theorem}{Theorem}[section]
\newtheorem{prop}[theorem]{Proposition}
\newtheorem{lemma}[theorem]{Lemma}
\newtheorem{cor}[theorem]{Corollary}

\theoremstyle{definition}
\newtheorem{defin}[theorem]{Definition}

\newtheorem{remark}[theorem]{Remark}

\numberwithin{equation}{section}

\usepackage{amsthm}
\usepackage[arrow,curve,matrix]{xy}
\usepackage{amsmath}  
\usepackage{amssymb}
\usepackage{graphicx}
\usepackage{bbm}
\usepackage{mathrsfs}
\usepackage{amscd}
\usepackage[utf8x]{inputenc}
\usepackage{a4wide}

\DeclareMathOperator{\Stab}{Stab}
\DeclareMathOperator{\Spec}{Spec}
\DeclareMathOperator{\pr}{pr}
\DeclareMathOperator{\id}{id}
\DeclareMathOperator{\sgn}{sgn}
\DeclareMathOperator{\Hom}{Hom}

\DeclareMathOperator{\ihom}{\underline{\mathcal{H}\textit{om}}}
\DeclareMathOperator{\iext}{\underline{\mathcal{E}\textit{xt}}}
\DeclareMathOperator{\sHom}{\mathcal{H}\textit{om}}
\DeclareMathOperator{\Ext}{Ext}
\DeclareMathOperator{\ext}{ext}
\DeclareMathOperator{\Coh}{Coh}
\DeclareMathOperator{\QCoh}{QCoh}
\DeclareMathOperator{\sExt}{\mathcal{E}\textit{xt}}
\DeclareMathOperator{\D}{D}
\DeclareMathOperator{\Ho}{H}
\DeclareMathOperator{\Hyp}{H}
\DeclareMathOperator{\For}{For}
\DeclareMathOperator{\Inf}{Inf}
\DeclareMathOperator{\Res}{Res}
\DeclareMathOperator{\Hilb}{Hilb}
\DeclareMathOperator{\GHilb}{GHilb}
\DeclareMathOperator{\glob}{\Gamma}
\DeclareMathOperator{\Aut}{Aut}
\DeclareMathOperator{\Tor}{Tor}
\DeclareMathOperator{\sTor}{\mathfrak{T}\textit{or}}
\DeclareMathOperator{\codim}{codim}
\DeclareMathOperator{\GL}{GL}

\DeclareMathOperator{\supp}{supp}
\DeclareMathOperator{\Pic}{Pic}
\DeclareMathOperator{\Kom}{Kom}

\DeclareMathOperator{\triv}{triv}
\DeclareMathOperator{\Yon}{Yon}
\DeclareMathOperator{\tot}{tot}
\DeclareMathOperator{\Mod}{Mod}
\DeclareMathOperator{\im}{im}

\newcommand{\alt}{\mathfrak a}
\newcommand{\sym}{\mathfrak S}
\newcommand{\C}{\mathbbm C}
\newcommand{\N}{\mathbbm N}
\newcommand{\Q}{\mathbbm Q}

\newcommand{\Z}{\mathbbm Z}

\newcommand{\F}{\mathcal F}
\newcommand{\E}{\mathcal E}

\newcommand{\M}{\mathcal M}
\newcommand{\Fb}{\mathcal F^\bullet}
\newcommand{\Eb}{\mathcal E^\bullet}
\newcommand{\Gb}{\mathcal G^\bullet}
\newcommand{\Ab}{\mathcal A^\bullet}
\newcommand{\m}{\mathfrak m}
\newcommand{\dv}{\mathbbm v}
\newcommand{\nax}{\mathfrak n}
\newcommand{\reg}{\mathcal O}
\newcommand{\I}{\mathcal I}
\newcommand{\n}{[n]}
\newcommand{\eps}{\varepsilon}
\renewcommand{\theta}{\vartheta}
\renewcommand{\rho}{\varrho}
\renewcommand{\phi}{\varphi}
\renewcommand{\_}{\underline{\,\,\,\,}}
\begin{document}

\title{Extension groups of tautological sheaves on Hilbert schemes}
\author{Andreas Krug}
\address{Universität Augsburg, 
Lehrstuhl für Algebra und Zahlentheorie}
\email{andreas.krug@math.uni-augsburg.de}
\begin{abstract}
 We give formulas for the extension groups between tautological sheaves and more general between tautological objects twisted by a determinant line bundle on the Hilbert scheme of points on a smooth quasi-projective surface. We do this using L. Scala's results about the image of tautological sheaves under the Bridgeland--King--Reid equivalence. We also compute the Yoneda products in terms of these formulas.
\end{abstract}
\maketitle
\section{Introduction}
For every smooth quasi-projective surface $X$ over $\C$ there is a series of associated higher dimensional smooth varieties namely the \textit{Hilbert schemes of $n$ points on $X$} for $n\in\N$. They are the fine moduli spaces $X^{[n]}$ of zero dimensional subschemes of length $n$ of $X$.
Thus, there is a universal family $\Xi$ together  with its projections
\[X\overset{\pr_X}\gets\Xi\overset{\pr_{X^{[n]}}}\to X^{[n]}\,.\]
Using this, one can associate to every coherent sheaf $F$ on $X$ the so called \textit{tautological sheaf} $F^{[n]}$ on each $X^{[n]}$ given by  
\[F^{[n]}:=\pr_{X^{[n]}*}\pr_X^* F\,.\]
More generally, for any object of the bounded derived category $F^\bullet\in\D^b(X)$ using the Fourier-Mukai transform with kernel the structural sheaf of $\Xi$ yields the \textit{tautological object}
\begin{align*}(F^\bullet)^{[n]}:=\Phi^{X\to X^{[n]}}_{\reg_\Xi}(F^\bullet)\,.\end{align*}
It is well known (see \cite{Fog}) that the Hilbert scheme $X^{[n]}$ of $n$ points on $X$ is a resolution of the singularities of $S^nX=X^n/\sym_n$ via the \textit{Hilbert-Chow morphism} 
\[\mu\colon X^{[n]}\to S^nX\quad,\quad \xi\mapsto\sum_{x\in\xi}\ell(\xi,x)\cdot x\,.\]
For every line bundle $L$ on $X$ the line bundle $L^{\boxtimes n}\in \Pic(X^n)$ descents to the line bundle $(L^{\boxtimes n})^{\sym_n}$ on $S^nX$. Thus, for every $L\in\Pic(X)$ there is the \textit{determinant line bundle} on $X^{[n]}$ given by 
\[\mathcal D_L:=\mu^*((L^{\boxtimes n})^{\sym_n})\,.\]
One goal in studying Hilbert schemes of points is to find formulas expressing the invariants of $X^{[n]}$ in terms of the invariants of the surface $X$. This includes the invariants of the induced sheaves defined above. There are already some results in this area.
For example, in \cite{Leh} there is a formula for the Chern classes of $F^{[n]}$ in terms of those of $F$ in the case that $F$ is a line bundle. In \cite{BNW} the existence of 
universal formulas, i.e. formulas independent of the surface $X$, expressing the characteristic classes of any tautological sheaf in terms of the characteristic classes of $F$ is shown and those formulas are computed  in some cases.
Furthermore, Danila (\cite{Dandual}, \cite{Dan}, \cite{Danglob}) and Scala (\cite{Sca1}, \cite{Sca2}) proved formulas for the cohomology of tautological sheaves, determinant line bundles, and some natural constructions (tensor, wedge, and symmetric products) of these. 
In particular, in \cite{Sca2} there is the formula 
\begin{align}\label{Scaformula}\Ho^*(X^{[n]},F^{[n]}\otimes \mathcal D_L)\cong \Ho^*(F\otimes L)\otimes S^{n-1}\Ho^*(L)\end{align}
for the cohomology of a tautological sheaf twisted by a determinant line bundle. 
In this paper we compute extension groups between tautological sheaves and more generally \textit{twisted tautological objects}, i.e. tautological objects tensorised with determinant line bundles. 
Our main theorem is the existence of natural isomorphisms of graded vector spaces 
\begin{align*}
\Ext^*((E^\bullet)^{[n]}\otimes \mathcal D_L,(F^\bullet)^{[n]}\otimes \mathcal D_M)&\cong \begin{aligned} &\Ext^*(E^\bullet\otimes L,F^\bullet\otimes M)\otimes S^{n-1}\Ext^*(L,M)\oplus\\ &\Ext^*(E^\bullet\otimes L,M)\otimes \Ext^*(L,F^\bullet\otimes M)\otimes S^{n-2}\Ext^*(L,M)\end{aligned}
\end{align*}
for objects $E^\bullet,F^\bullet\in\D^b(X)$ and line bundles $L,M\in \Pic(X)$. We also give a similar formula for $\Ext^*((E^\bullet)^{[n]}\otimes \mathcal D_L,\mathcal D_M)$. Since $\mathcal D_{\reg_X}=\reg_{X^{[n]}}$, by setting $L=M=\reg_X$ the extension groups and the cohomology of the dual of non-twisted tautological bundles occur as special cases. As an application of the formula for the extension groups we show for $X$ a projective surface with a trivial canonical bundle that twisted tautological objects are never spherical or $\mathbbm P^n$-objects in $\D^b(X^{[n]})$.

We will use Scala's approach of \cite{Sca1}, which in turn uses the Bridgeland--King--Reid equivalence. 
Let $G$ be a finite group acting on a smooth quasi-projective variety $M$. A \textit{G-cluster} on $M$ is a zero-dimensional closed $G$-invariant subscheme $Z$ of $M$ where $\glob(Z,\reg_Z)$ equipped with the induced $G$-action is isomorphic to the natural representation $\C^G$.
Bridgeland, King and Reid proved in \cite{BKR} that under some requirements the irreducible component $\Hilb^G(M)$, called the \textit{Nakamura $G$-Hilbert scheme}, of the fine moduli space of $G$-clusters on $M$ which contains the points corresponding to free orbits is a crepant resolution of the quotient $M/G$. Furthermore, they showed that the $G$-equivariant Fourier-Mukai transform with kernel the structural sheaf of the universal family $\mathcal Z$ of $G$-clusters 
\[\Phi:=\Phi_{\reg_{\mathcal Z}}\colon \D^b(\Hilb^G(M))\to \D^b_G(M)\]
between the bounded derived category of $\Hilb^G(M)$ and the equivariant bounded derived category $\D_G^b(M)=\D^b(\Coh_G(M))$ of $M$ is an equivalence of triangulated categories. Hence, $\Phi$ is called the 
\textit{Bridgeland--King--Reid equivalence}. 
Haiman proved in \cite{Hai} that $X^{[n]}$ is isomorphic as a resolution of $S^nX$ to $\Hilb^{\sym_n}(X^n)$ with the \textit{isospectral Hilbert scheme}
\[I^nX:=\bigl(X^{[n]}\times_{S^nX} X^n\bigr)_{\text{red}}\]
as the universal family of $\sym_n$-clusters. Furthermore, the conditions of the Bridgeland--King--Reid theorem are satisfied in this situation. Thus, there is the equivalence
\[\Phi:=\Phi_{\reg_{I^nX}}^{X^{[n]}\to X^n}\colon \D^b(X^{[n]})\xrightarrow\simeq \D^b_{\sym_n}(X^n)\,.\]
In general, for $\Eb,\Fb\in \D^b(\mathcal A)$ objects in the derived category of any abelian category with enough injectives and every $k\in \Z$ there is the identification  $\Ext^k(\Eb,\Fb)\cong \Hom_{\D^b(\mathcal A)}(\Eb,\Fb[k])$. Using this, we can compute extension groups on $X^{[n]}$ as $\sym_n$-invariant extension groups on $X^n$ 
after applying the Bridgeland-King-Reid equivalence, i.e. for $\Eb,\Fb\in\D^b(X^{[n]})$ we have
\[\Ext^*_{X^{[n]}}(\Eb,\Fb)\simeq \sym_n\Ext^*_{X^n}(\Phi(\Eb),\Phi(\Fb))\,.\]
Using the quotient morphism $\pi\colon X^n\to S^nX$ the right-hand side can be rewritten further as
\begin{align*}
\sym_n\Ext^i_{X^n}(\Phi(\E^\bullet),\Phi(\F^\bullet)))
\cong \Ho^i(S^nX,[\pi_*R\sHom_{X^n}(\Phi(\Eb),\Phi(\Eb)]^{\sym_n})\,. 
\end{align*}
So instead of computing the extension groups of tautological sheaves and objects on $X^{[n]}$ directly, the approach is to compute them for the image of these objects under the Bridgeland--King--Reid equivalence.  
In order to do this we need a good description of $\Phi(F^{[n]})\in \D^b_{\sym_n}(X^n)$ for $F^{[n]}$ a tautological sheaf. This was provided by Scala in \cite{Sca1} and \cite{Sca2}. He showed that $\Phi(F^{[n]})$ is always concentrated in degree zero. This means that we can replace $\Phi$ by its non-derived version $p_*q^*$ where $p$ and $q$ are the projections from $I^nX$ to $X^n$ and $X^{[n]}$ respectively, i.e. we have $\Phi(F^{[n]})\simeq p_*q^*(F^{[n]})$. Moreover, he gave for $p_*q^*(F^{[n]})$ a right resolution $C^\bullet_F$. This is a $\sym_n$-equivariant complex associated to $F$ concentrated in non-negative degrees whose terms are of the form
\[C^0_F=\bigoplus_{i=1}^n \pr_i^* F\quad,\quad
C^p_F=\bigoplus_{I\subset [n],\,|I|=p+1} F_I \text{  for $p>0$.}
\]
Here, $F_{I}$ denotes the sheaf $\iota_{I*}p_I^*F$, where $\iota_{I}\colon \Delta_{I}\to X^n$ is the inclusion of the partial diagonal and 
$p_{I}\colon \Delta_{I}\to X$ is the projection induced by the projection $\pr_i\colon X^n\to X$ for any $i\in I$.
 
The main step of the proof of the main formula is to compute $[\pi_*R\sHom_{X^n}(\Phi(E^{[n]}),\Phi(F^{[n]})]^{\sym_n}$ in the case of tautological bundles, i.e. for $E,F$ locally free sheaves on $X$. We show that  
\begin{align}\label{bundlehom}
[\pi_*R\sHom(\Phi(E^{[n]}),\Phi(F^{[n]}))]^{\sym_n}\simeq[\pi_*\sHom(C^0_E,C^0_F)]^{\sym_n}\,.           
\end{align}
In particular the $\sym_n$-invariants of the higher sheaf-Ext vanish.
The isomorphism in degree zero is shown using the fact that the supports of the terms $C^p_E$ for $p>0$ have codimension at least two and the normality of the variety $X^n$. For the vanishing of the higher derived sheaf-Homs we do computations in the spectral sequences associated to the bifunctor
\[\ihom(\_,\_):=[\pi_*\sHom_{X^n}(\_,\_)]^{\sym_n}\] and the complexes $C^\bullet_E$ and $C^\bullet_F$.    
We can generalise (\ref{bundlehom}) from locally free sheaves $E$ and $F$ to arbitrary objects $E^\bullet,F^\bullet\in \D^b(X)$ by taking locally free resolutions on $X$ and using some formal arguments for derived functors. Using the fact that $\Phi(\F^\bullet\otimes \mathcal D_L)\simeq \Phi(\F^\bullet)\otimes L^{\boxtimes n}$ holds for every object $\F^\bullet\in \D^b(X^{[n]})$ we can generalise further to twisted tautological objects and get
\begin{align*}
\left[\pi_*R\sHom(\Phi((E^\bullet)^{[n]}\otimes \mathcal D_L),\Phi((F^\bullet)^{[n]}\otimes \mathcal D_M))\right]^{\sym_n}&\simeq \left[\pi_*R\sHom(C^0_{E^\bullet}\otimes L^{\boxtimes n},C^0_{F\bullet}\otimes M^{\boxtimes n})\right]^{\sym_n}\,.
\end{align*}
Then we can compute the desired extension groups as the cohomology of the object on the right and get the main theorem.
There is also a similar formula for the extension groups between two determinant line bundles and we can apply our arguments to generalise Scala's result to get a formula for the cohomology of twisted tautological objects.
So now there are formulas for $\Ext^*_{X^{[n]}}(\Fb,\Gb)$ whenever both of $\Fb$ and $\Gb$ are either twisted tautological objects or determinant line bundles. In the last section we describe how to compute all the possible Yoneda products in terms of these formulas.

The author wants to thank his adviser Marc Nieper-Wi\ss kirchen for his many valuable suggestions. He thanks Malte Wandel for interesting and helpful discussions about the topic of this paper. 
\section{Notations and conventions}\label{not}
\begin{itemize}
 \item Given an abelian category $\mathcal A$ we will write $\cong$ for isomorphisms in $\mathcal A$ and $\simeq$ for isomorphisms in the derived category $\D(\mathcal A)$. If an object $A\in\mathcal A$ shows up on one side of the sign $\simeq$ it is considered as the complex with $A$ in degree 0 and vanishing terms elsewhere. If we want to emphasize that an isomorphism is $G$-equivariant we will sometimes write $\cong_G$ respectively $\simeq_G$.
\item For an exact functor $F\colon \mathcal A\to \mathcal B$ between abelian categories we write again $F$ for the induced functor $\D(\mathcal A)\to \D(\mathcal B)$ on the  level of the derived categories. If we want to emphasize that the image of a complex $A^\bullet\in\D^b(\mathcal A)$ under $F$ is computed by applying $F\colon\mathcal A\to \mathcal B$ term-wise we write $F(A)^\bullet$ instead of $F(A^\bullet)$. We also sometimes write $F(A^\bullet)$ in formulas for object in the derived categoies when $F$ is not exact although in this case $F$ is not a functor between the derived categories. This means that we apply the functor $\Kom(F)$ to the complex $A^\bullet$ and consider it again as an object in the derived category afterwards.
\item
Let $C^\bullet\in \D^b(\mathcal A)$ for any abelian category $\mathcal A$. We write $\mathcal H^i(C^\bullet)$ for the $i$-th cohomology of a complex, i.e. $\mathcal H^i(C^\bullet)=\ker(d^i)/\im(d^{i-1})$. If $\mathcal A=\Coh(X)$ is the category of coherent sheaves on a scheme $X$, we write $\Ho^i(X,C^\bullet)$ for the $i$-th (hyper-)cohomology of the complex of sheaves, i.e. $\Ho^i(X,\_)=R^i\Gamma(X,\_)$. 
We will often drop the $X$ in the notation, i.e. we write $\Ho^i(\_):=\Ho^i(X,\_)$, especially when $X$ is a fixed smooth quasi-projective surface.
\item On a scheme $X$ the sheaf-Hom functor is denoted by $\sHom_{\reg_X}$, $\sHom_X$ or just $\sHom$.
We write $(\_)^\vee=\sHom(\_,\reg_X)$ for the operation of taking the dual of a sheaf and $(\_)^\dv=R\sHom(\_,\reg_X)$ for the derived dual.
 \item All varieties are reduced and irreducible.
 \item Graded vector spaces are denoted by $V^*:=\oplus_{i\in\Z} V^i[-i]$. The symmetric power of a graded vector space is taken in the graded sense. That means that $S^nV^*$ are the coinvariants of $(V^*)^{\otimes n}$ under the $\sym_n$-action given on homogeneous vectors by
\[\sigma(u_1\otimes\dots\otimes u_k):=\eps_{\sigma,p_1,\dots,p_k}(u_{\sigma^{-1}(1)}\otimes\dots\otimes u_{\sigma^{-1}(k)})\,.\]
Here the $p_i$ are the degrees of the $u_i$ and the sign $\eps_{\sigma,p_1,\dots,p_k}$ is defined by setting $\eps_{\tau,p_1,\dots,p_k}=(-1)^{p_i\cdot p_{i+1}}$ for the transposition $\tau=(i\,,\,i+1)$ and requiring it to be a homomorphism in $\sigma$. Since the graded vector spaces we will consider are defined over $\C$, the coinvariants under this action coincide with the invariants under the isomorphism
\[ u_1\cdots u_n\mapsto \frac1{n!}\sum_{\sigma\in\sym_n} \eps_{\sigma,p_1,\dots,p_k}(u_{\sigma^{-1}(1)}\otimes\dots\otimes u_{\sigma^{-1}(k)})\,.\]
\item For every positive integer $n\in\N$ we set $\n:=\{1,\dots,n\}$. For any subset $I\subset \n$ we denote by $\bar I=\n\setminus I$ its complement in $\n$.
We denote by $\sym_I$ the group of permutations of the set $I$ and consider it as a subgroup of $\sym_n=\sym_{[n]}$. For better readability we sometimes write
$\overline{\sym_I}$ instead of $\sym_{\bar I}$. 
\item We will often write the symbol ``PF`` above an isomorphism symbol to indicate that the isomorphism is given by the projection formula.
 Also we use ``lf`` when the given isomorphism is because of some sheaf being locally free and thus a tensor product or sheaf-Hom functor needs not to be derived  or commutes  with taking cohomology.
\item Let $\iota\colon Z\to X$ be a closed embedding of schemes and let $F\in \QCoh(X)$ be a quasi-coherent sheaf on $X$. The symbol $F_{\mid Z}$ will sometimes denote the sheaf $\iota^*F\in \QCoh(Z)$ and at other times the sheaf $\iota_*\iota^*F\in\QCoh(X)$. The restriction morphism \[F\to F_{\mid Z}=\iota_*\iota^*F\] is the unit of the adjunction $(\iota^*, \iota_*)$. The image of a section $s\in F$ under this morphism is denoted by $s_{\mid Z}$.   
\end{itemize}

\section{Equivariant sheaves}
\subsection{Basic definitions}
Let $G$ be a finite group acting on a scheme $X$.
All group actions will be left actions.
Let $F\in \QCoh(X)$ be a quasi-coherent sheaf. A \textit{$G$-linearization} on $F$ is a family of $\reg_X$-linear isomorphisms
$(\lambda_g\colon F\to g^*F)_{g\in G}$ with the properties $\lambda_e=\id$ and 
\[\lambda_{hg}= g^*\lambda_h\circ \lambda_g\colon F\to g^*h^*F\cong(hg)^*F\,.\] 
A quasi-coherent sheaf $F\in \QCoh(X)$ together with a $G$-linearization is called a \textit{$G$-equivariant sheaf} or just a \textit{$G$-sheaf}.
For two $G$-sheaves $(E,\lambda)$, $(F,\mu)$ the group $G$ acts on $\Hom_{\reg_X}(E,F)$ by 
\[ g\cdot\phi:=\mu_{g^{-1}}^{-1}\circ ((g^{-1})^*\phi)\circ \lambda_{g^{-1}}\,.\]
A \textit{morphism of $G$-sheaves} is an $\reg_X$-linear morphism of the underlying ordinary sheaves which is invariant under this action, i.e.
\[G\Hom_{\reg_X}(E,F):=(\Hom_{\reg_X}(E,F))^G=\{\phi\colon E\to F\mid \mu_g\circ \phi=(g^*\phi)\circ \lambda_g\,\forall g\in G\}\,.\]
This gives the abelian category $\QCoh_G(X)$ of $G$-equivariant sheaves on $X$. The full abelian subcategory of $G$-sheaves 
whose underlying sheaves are coherent is denoted by
 $\Coh_G(X)$.
\subsection{Inflation and Restriction}
For every subgroup $H\subset G$ there is the \textit{restriction functor} $\Res_G^H\colon \QCoh_G(X)\to \QCoh_H(X)$ given by restricting the $G$-linearization of a $G$-sheaf $F$ to a $H$-linearization. In the special case $H=1$ we call the functor
\[\For:=\Res_G^1\colon \QCoh_G(X)\to \QCoh(X)\]
also the \textit{forgetful functor}. We choose a system of representatives of $H\setminus G$. Then there is the \textit{inflation functor}
\[\Inf^G_H\colon \QCoh_H(X)\to\QCoh_G(X)\quad,\quad \Inf_H^G(E,\lambda):= \bigoplus_{[g]\in H\setminus G} g^*E\]
depending only by canonical isomorphism on the chosen representatives.
The $G$-linearization $\mu$ of $\Inf_H^G(E)$ is given as follows.
For every $\sigma\in G$ and $g$ one of the chosen representatives of the classes in $H\setminus G$ there exist uniquely $h\in H$ and $\hat g$ one of the 
chosen representatives such that $h\cdot \hat g= g \cdot \sigma$. 
Now for a local section $s=(s_g\in g^*E)_{[g]\in H\setminus G}$ we define 
$\mu_\sigma(s)_g\in \sigma^*g^* E$ as the image of $s_{\hat g}$ under the isomorphism
\[ \hat g ^* E\overset{\hat g ^*\lambda_h}{\to} \hat g^*h^*E\cong (h\hat g)^* E=(g\sigma)^* E\cong \sigma^*g^* E\,.\]
The inflation functor is left adjoint to the restriction functor. The adjoint pair
\[\Inf^G_H\colon \QCoh_H(X) \rightleftharpoons \QCoh_G(X)\colon \Res_G^H\]
clearly restricts to an adjoint pair of the full subcategories of equivariant coherent sheaves. Both the inflation and the restriction functor are exact. 
Let $G$ act transitively on a set $I$. 
Let 
$\M$ be a $G$-sheaf on $X$ where the underlying ordinary sheaf admits a decomposition $\M=\oplus_{i\in I} \M_i$ such that for any $i\in I$ and $g\in G$ the linearization $\lambda$ restricted to $\M_i$ is an isomorphism $\lambda_g\colon \M_i\overset\cong\to g^*\M_{g(i)}$. Then the $G$-linearization of $\M$ restricts to a  $\Stab_G(i)$-linearization of $\M_i$.
\begin{lemma}\label{inf}\label{infrem}
 Under these assumptions for every  $i\in I$ we have $\M\cong_G \Inf_{\Stab(i)}^G\M_i$.
\end{lemma}
\begin{proof}
 We only have to show that $\M$ fulfills the adjointness property of $\Inf_{\Stab(i)}^G(\M_i)$. Let $(\F,\mu)$ be a $G$-sheaf and $\phi\colon \M\to \F$ a $G$-equivariant morphism. For $j\in I$ let $\phi_j\colon \M_j\to \F$ be the component of $\phi$. Choose a $g\in G$ with $g(j)=i$. Then by the $G$-equivariance
\[\phi_j= \mu_g^{-1}\circ(g^*\phi_i)\circ (\lambda_{g|\M_j})\,.\]
Thus $\phi$ is determined by the $\Stab(i)$-equivariant morphism $\phi_i$. On the other hand every given $\Stab(i)$-equivariant morphism 
$\phi_i\colon \M_i\to\F$ gives rise to a $G$-equivariant $\phi\colon \M\to \F$ by taking the above equation as the definition of the other components.     
\end{proof}
\begin{remark}\label{conv}
Conversely the inflated sheaf as defined above can be written in the form of $\M$ from above by setting $I=H\setminus G$ with $\sigma\in G$ acting on $I$ by $\sigma(Hg)=Hg\sigma^{-1}$.
\end{remark} 
\subsection{Schemes with trivial $G$-action}\label{triv} 
If $G$ acts trivially on $X$, a $G$-linearization of $F$ is just a $G$-action on $F$, i.e. a family of $G$-actions on $F(U)$ for each open subset $U\subset X$,
which are compatible with the restrictions. Thus, in this case there is also the \textit{functor of taking invariants}
$[\_]^G\colon \QCoh_G(X)\to \QCoh(X)$ given by $F^G(U):=[F(U)]^G$. If the scheme is defined over a field of characteristic zero the functor of taking invariants is exact because of the existence of the Reynolds operator. Considering an ordinary sheaf on $X$ as a sheaf with the trivial $G$-action also yields an exact functor $\triv\colon \QCoh(X)\to\QCoh_G(X)$.
Togethter these functors form the adjoint pair
\[\triv \colon \QCoh(X) \rightleftharpoons \QCoh_G(X)\colon [\_]^G\]
 Again both functors restrict to functors between the coherent categories. We will write interchangeably $[\_]^G$ and $(\_)^G$ or leave the brackets away. 
\subsection{Equivariant geometric functors}\label{geomfun}
Let $G$ act also on another scheme $Y$ and $f\colon X\to Y$ be an $G$-equivariant morphism.
Then there are the \textit{equivariant geometric functors} \[\otimes,\sHom\colon \QCoh_{G}(X)\times\QCoh_{G}(X)\to \QCoh_G(X)\,,\]
 $f_*\colon \QCoh_G(X)\to\QCoh_G(Y)$ and 
$f^*\colon \QCoh_G(Y)\to \QCoh_G(X)$ which are defined in a canonical way such that they are compatible with the 
corresponding non-equivariant functors via the forgetful functor. 
More concretely, since $f$ commutes with the action of every $\sigma\in G$ for $D,E\in\QCoh(X)$ and $F\in \QCoh(Y)$ there are natural isomorphisms $f^*\sigma^* F \cong \sigma^*f^* F$, $f_*\sigma^*E\cong \sigma^*f_*E$ (flat base change) and $\sigma^*D\otimes \sigma^*E\cong \sigma^*(D\otimes E)$. This allows us to define the $G$-linearization of the pull-back, the push-forward and the tensor product as the pull-back, push-forward and tensor product of the $G$-linearization of the original sheaves. Since the morphism $\sigma$ is an automorphism, the natural morphism $\sigma^*\sHom(D,E)\to \sHom(\sigma^*D,\sigma^*E)$ is also an isomorphism.
Using this identification, the linearization $\alpha$ of $\sHom(D,E)$ is given by $\alpha_\sigma(\phi)=(\mu_{\sigma|U})\circ \phi\circ(\lambda^{-1}_{\sigma|U})$ where $\phi\in \sHom(D,E)(U)$ for an open subset $U\subset X$ and $\lambda$ and $\mu$ are the linearizations of $D$ and $E$.   
Because of the compatibility with the forgetful functor, all these functors restrict to functors between the 
full subcategories of coherent sheaves if and only if the corresponding non-equivariant geometric functors do. Every $G$-linearization induces a $G$-action on the global sections of $X$.   
Hence, there is also the functor \[\glob(X,\_)\colon \QCoh(X)\to \Mod(\Z[G])\] mapping to the category of abelian groups equipped with an additive $G$-action. If $X$ is defined over a field $k$ we always assume the structural morphism $X\to \Spec k$ to be $G$-invariant, i.e $G$ acts $k$-linear on $X$. Thus, in this case the global sections functor factorises  over the category of $k[G]$-modules. We have the formula
\[G\Hom(D,E)=[\glob(X,\sHom(D,E))]^G\,.\]
\begin{lemma}\label{inffun}
Let $H\le G$ be a subgroup, $\E$ a $G$-sheaf and $\F$ a $H$-sheaf on $X$. We set $\Inf=\Inf_H^G$ and $\Res=\Res_G^H$. Then
\begin{align*}\E\otimes\Inf\F\cong_G \Inf(\Res(\E)\otimes \F)\,&,\, \sHom(\E,\Inf \F)\cong_G \Inf(\sHom(\Res \E,\F))\,,\\ \sHom(\Inf \F,\E)&\cong_G \Inf(\sHom(\F,\Res \E))\,.
\end{align*}
\end{lemma}
\begin{proof}
The underlying sheaf of $\E\otimes \Inf\F$ is $\oplus_{[g]\in H\setminus G} \E\otimes g^*\F$. For $\sigma\in G$ the $\sigma$-linearization of $\E\otimes \Inf \F$ maps $\E\otimes \hat g^* F$ isomorphic to $\sigma^*(\E\otimes g^*\F)$, where $[g]$ is the image of $[\hat g]$ under the action of $\sigma $ on $H\setminus G$ defined in remark \ref{conv}. Furthermore the $H$-linearization on the component $\E\otimes\F$ coincides with the $H$-linearization of $\Res(\E)\otimes \F$. By lemma \ref{inf} this gives the result. The other two cases are proven similarly.    
\end{proof}
Also the equivariant pull-backs and push-forwards commute with the inflation functor.
\begin{remark}\label{induequi}
Let $X$ be a $G$-scheme and $\pi\colon X\to Y$ a $G$-invariant morphism, e.g. the quotient morphism. Then there is the composition 
\[\pi^*\circ \triv\colon \QCoh(Y)\to \QCoh_G(X)\quad,\quad F\mapsto (\pi^*F,\lambda)\,.\]
For $\sigma\in G$ we denote $\lambda_\sigma$ in this case also by $\sigma_*$. The map $\sigma_*(U)\colon (\pi^*F)(U)\to (\pi^*F)(\sigma(U))$ is given for $s\in (\pi^*F)(U)$ by $(\sigma_*(U)s)(x)=s(\sigma^{-1}x)$ for every ring $R$ and every $R$-valued point $x$ of $\sigma(U)$. In the later sections, mostly $F$ will be a locally free sheaf and $X$ a variety over $\C$. In this case it suffices to consider $\C$-valued points $x\in \sigma(U)$.
\end{remark}
\subsection{Derived equivariant categories}
Let from now on $X$ be a noetherian scheme. We denote the derived categories as \[\D_G(\QCoh(X)):=\D(\QCoh_G(X))\,,\, 
\D_G(X):=\D_G(\Coh(X)):=\D(\Coh_G(X))\,.\]
As usual for $*=+,-,b$ we denote by $\D^*_G(\QCoh(X))$ respectively $\D^*_G(X)$ the full subcategories of complexes with bounded from below, bounded from above or bounded cohomology. For the same reasons as in the non-equivariant case (see e.g. \cite[Prop. 3.5]{Huy}) the category $\D^*_G(X)$ can be identified with the full subcategory of $\D^*_G(\QCoh(X))$ of complexes with coherent cohomology sheaves. Since they are exact, the functors $\Inf$, $\Res$, and in case of the trivial $G$-action also $\triv$ from above define functors on the derived categories. If the scheme $X$ equipped with the trivial $G$-action is defined over a field $k$ of characteristic zero, the functor of taking invariants also is defined between the derived categories. For an $k$-scheme $X$ with an arbitrary $G$-action let $\pi\colon X\to Y$ be a $G$-invariant morphism of $k$-schemes, i.e. an equivariant morphism when considering $Y$ with the trivial action. Then we can push forward $G$-sheafs  on $X$ along $\pi$ and take the $G$-invariants on $Y$ afterwards. If $\pi_*$ is an exact functor (which is independent of considering the equivariant functor $\QCoh_G(X)\to \QCoh_G(Y)$ or the non-equivariant one $\QCoh(X)\to \QCoh(Y)$ because of the compatibility with the forgetful functor), we write for short $[\_]^G$ instead of $[\_]^G\circ \pi_*$, which yields also directly a functor on the derived categories.
\subsection{Injective and locally free sheaves}
There are always enough injectives in the category $\QCoh_G(X)$ (see \cite[section 5.1]{Gro1}).
\begin{lemma} Whenever a $G$-sheaf $(F,\lambda)\in\QCoh_G(X)$ is an injective object also its underlying sheaf $F\in\QCoh(X)$ is injective.
If $X$ is defined over a field of characteristic zero also the converse holds. 
\end{lemma}
\begin{proof} Let $(F,\lambda)\in\QCoh_G(X)$ be injective and
\[0\to E'\to E \to E''\to 0 \]
a short exact sequence in $\QCoh(X)$. Applying $\Inf$ to this sequence yields a short exact sequence in $\QCoh^G(X)$. By the adjointness of 
$\Inf$ and $\For$ we get  the following isomorphism of complexes:
\[\begin{xy}
\xymatrix{ 0 & \Hom(E',F)\ar[d]^\cong\ar[l] & \Hom(E,F)\ar[d]^\cong\ar[l] & \Hom(E'',F) \ar[l]\ar[d]^\cong & 0 \ar[l]
       \\  0 & G\Hom(\Inf E',F)\ar[l] & G\Hom(\Inf E,F)\ar[l] & G\Hom(\Inf E'',F)\ar[l] & 0 \ar[l]\,.}
\end{xy}\]  
Since the lower complex is exact by the injectivity of $F$ as an $G$-sheaf, the upper one is also. This shows the injectivity of $F\in \QCoh(X)$.
Taking invariants of $G$-representations over a field of characteristic zero is exact. Since $G\Hom(E,F)$ is defined as the space of invariants of $\Hom(E,F)$ we get the second statement.
\end{proof}
A \textit{$G$-equivariant locally free sheaf} is just a 
coherent $G$-sheaf whose underlying ordinary sheaf is locally free. 
\begin{lemma}
 There are enough $G$-equivariant locally free sheaves in $\Coh_G(X)$ if and only if there are enough locally free sheaves in $\Coh(X)$ 
\end{lemma}
\begin{proof}
Let there be enough locally free $G$-sheaves in $\Coh_G(X)$ and let $F\in\Coh(X)$. Then there is a locally free $G$-sheaf $E$ and a surjection 
$\phi\colon E\to \Inf(F)$. The component $\phi_e\colon E\to F$ is then a surjection in $\Coh(X)$. Let on the other hand $\Coh(X)$ have enough locally free sheaves and consider $F\in \Coh_G(X)$. Then there is a surjection $\psi\colon E\to \For(F)$ from a locally free sheaf $E\in \Coh(X)$. By the adjoint property of the inflation functor there is a map $\Inf(E)\to F$ in $\Coh_G(X)$. By construction it is still surjective.
\end{proof}
\subsection{Derived equivariant functors}    
Because of the existence of enough injective equivariant sheaves the functors $\glob(X,\_)$, $f_*$ and $\sHom$ from subsection \ref{geomfun} can be derived to (bi-)functors
\begin{align*}R\glob(X,\_)\colon\D^+_G(\QCoh(X))\to \D^+(\Mod(\Z[G])\,,\,  Rf_*\colon \D^+_G(\QCoh(X))\to \D^+_G(\QCoh(Y))\,,\,\\
 R\sHom\colon \D^-_G(\QCoh(X))^{\circ}\times \D^+_G(\QCoh(X))\to D^+_G(\QCoh(X))\,.
\end{align*}
If there are enough locally free coherent sheaves on $Y$ also the derived equivariant functors 
\[Lf^*\colon \D_G^-(Y)\to \D_G^-(X)\quad,\,\quad \_\otimes^L\_\colon \D_G^-(Y)\times \D_G^-(Y)\to \D_G^-(Y) \]
exist. Since the forgetful functor maps injective to injective and locally free to locally free sheaves, all the equivariant derived functors are compatible 
with their non-equivariant versions via the forgetful functor, for example the following diagram commutes
\[ \begin{CD}
\D^+_G(\QCoh(X))
@>{Rf_*}>>
\D^+_G(\QCoh(Y)) \\
@V{\For}VV
@VV{\For}V \\
\D^+(\QCoh(X))
@>>{Rf_*}>
\D^+(\QCoh(Y))\,.
\end{CD} \]
This implies that (when there are enough locally free sheaves) a derived geometric equivariant functor restricts to a functor between the bounded derived categories of coherent sheaves if and only if the corresponding non-equivariant functor does.  
Also the functor $G\Hom(\_,\_)$ of global $G$-homomorphisms can be derived. We define $G\Ext^i(\_,\_)$ to be the $i$-th derived functor of $G\Hom(\_,\_)$. It coincides with $\Hom_{\D(\QCoh^G(X))}(\_,\_[-i])$. 
All the common formulas (see e.g. the section ``Compatibilities`` in \cite[chapter 3]{Huy} ) relating the geometric derived functors generalise directly to the equivariant case. In the following lemma we proof one of them 
as an example. 
\begin{lemma}
Let $X$ be a scheme with a $G$-action such that there are enough locally free sheaves on $X$ and the geometric derived bifunctors restrict to
\[R\sHom\,,\,\otimes\colon \D^b_G(X)\times\D^b_G(X)\to \D^b_G(X)\,.\]
Then for every $\mathcal D^\bullet,\Eb,\Fb\in \D^b_G(X)$ there is a natural isomorphism 
\[R\sHom(\mathcal D^\bullet, \Eb)\otimes^L \Fb\simeq R\sHom(\mathcal D^\bullet, \Eb\otimes^L \Fb)\,.\]
\end{lemma}
\begin{proof}
Both sides are  computed by taking locally free resolutions and then applying the non-derived functors. Thus it suffices to show the formula for locally free equivariant sheaves and the non-derived functors. Let $(D,\lambda)$, $(E,\mu)$ and $(F,\nu)$ be $G$-equivariant locally free sheaves on $X$. It is well known that the map of ordinary sheaves
\[T\colon\sHom(D,E)\otimes F\to\sHom(D,E\otimes F)\quad,\quad \phi\otimes s\mapsto (d\mapsto \phi(d)\otimes s)\]
is an isomorphism. We only have to show that it is equivariant. We denote the linearization of the left-hand side by $\alpha$ and that of the right-hand side by $\beta$. Then for $g\in G$ indeed
\begin{align*} (g^*T)(\alpha_g(\phi\otimes s))=(g^*T)((\mu_g\circ\phi\circ\lambda_g^{-1})\otimes \nu_g(s))&= \left(d\mapsto(\mu_g\circ\phi\circ\lambda_g^{-1})(d)\otimes \nu_g(s)\right)\\
 &= \beta_g(d\mapsto \phi(d)\otimes s)\\
&= \beta_g(T(\phi\otimes s))\,. 
\end{align*}
\end{proof}
\subsection{Equivariant Grothendieck duality}\label{Verd}
Also Grothendieck duality generalises to proper $G$-equivariant morphisms (see \cite{Has}).
\begin{theorem} Let the finite group $G$ act on schemes $X$ and $Y$, which are of finite type over $\C$, and let $f\colon X\to Y$ be a proper $G$-equivariant morphism. Then there exists an exact functor $f^!\colon \D^+(\QCoh^G(Y))\to \D^+(\QCoh^G(X))$ which is compatible with the non-equivariant twisted inverse image functor via the forgetful functor and has the property that for every $\Fb\in\D^-(\QCoh^G(X))$ and $\Gb\in \D^+(\QCoh^G(Y))$ the following holds in $\D(\QCoh^G(Y))$:
\[Rf_*R\sHom_Y(\Fb,f^!\Gb)\simeq R\sHom_X(Rf_*\Fb, \Gb)\,. \]
\end{theorem}
We will need the case where $f=\iota\colon Z\to X$ is a regular embedding with vanishing ideal $\I_Z$ of a $G$-invariant subvariety.
In this case the normal sheaf
$N_{Z}:=(\iota^* \I_Z)^\vee$ is a locally free sheaf of rank equal to the codimension of $Z$ in $X$.
If $G$ acts on a scheme $X$, the structural sheaf always has a natural $G$-linearization given by $\lambda_g=(g^\#)^{-1}$, where $g^\#\colon g^*\reg_X\to \reg_X$ is the map of sheaves 
which belongs to the morphism $g\in\Aut(X)$. Now if $Z$ is a $G$-invariant closed subscheme with vanishing ideal $\I$, the map $g^\#$ restricts to
$g^\#\colon g^*\I\to \I$. Thus the sheaf $\I$ carries a natural $G$-linearization $\lambda$.  In the following we will always consider $\I_Z$ and $N_Z$ equipped with the linearization induced by $\lambda$.  
\begin{prop}\label{Gro}
Let $X$ be a variety over a field $k$ and $\iota\colon Z\to X$ be the immersion of a closed $G$-invariant local complete intersection subvariety of codimension $c$ with vanishing ideal $\I$ and $\Gb\in \D^b_G(\QCoh(X))$. 
Then there are canonical $G$-equivariant isomorphisms
\begin{enumerate}
 \item \[(\iota_*\reg_Z)^\dv\simeq \iota_*({\wedge}^c N_Z)[-c]\,,\]
 \item \[\sExt^k(\iota_*\reg_Z,\iota_*\reg_Z)\cong \iota_*({\wedge}^k N_Z)\quad \forall\, 0\le k\le c\,,\]
 \item \[\iota^!\Gb\simeq L\iota^*\Gb\otimes ({\wedge}^c N_Z)[-c]\,.\]
\end{enumerate}
That means, that the $G$-linearizations on the right sides of the formulas are all the ones canonically induced by the linearization $\lambda$ of $\I$. 
\end{prop}
\begin{proof}
The proposition is proved in chapter 28 of \cite{Has} in the more general framework of diagrams of schemes. How to obtain schemes with a group action as a special case is explained in the introduction and chapter 29.  
\end{proof}

\section{Preliminary lemmas}
\subsection{Derived bifunctors}
Let $\mathcal A$, $\mathcal B$ and $\mathcal C$ be abelian categories and 
$F\colon \mathcal A\times  \mathcal B\to \mathcal C$
be an additive bifunctor which is left exact in both arguments. The functor 
$K^+(F)\colon K^+(\mathcal A)\times K^+(\mathcal B)\to K^+(\mathcal C)$ is defined by 
\[K(F)(A^\bullet,B^\bullet):=F^\bullet(A^\bullet, B^\bullet) :=\tot F(A^\bullet,B^\bullet)\,.\]
We assume that there is a full additive subcategory $\mathcal J$ of $\mathcal B$ such that for every $B\in \mathcal J$ the functor
$F(\_,B)\colon \mathcal A\to \mathcal C$ is exact and for every $A\in\mathcal A$ the subcategory $\mathcal J$  is a $F(A,\_)$-adapted class.
Under these assumptions the right derived bifunctor
\[RF\colon \D^+(\mathcal A)\times \D^+(\mathcal B)\to \D^+(\mathcal C)\]
exists. Furthermore $K^+(\mathcal J)$ is a $K^+(F)(A^\bullet,\_)$-adapted class for every $A^\bullet\in K^+(\mathcal A)$. Thus also the right derived functor $R(K^+(F)(A^\bullet,\_))$ exists and we have for each $A^\bullet\in \mathcal A$ and $B^\bullet\in \mathcal B$  
\[RF(A^\bullet,B^\bullet)\simeq R(K^+(F)(A^\bullet,\_))(B^\bullet)\]
(see \cite[Section 13.4]{KS}). An example were the above assumptions are fulfilled is for any scheme $X$ the functor
\[\sHom\colon \QCoh(X)^\circ \times \QCoh(X)\to \QCoh(X)\,,\]
where we can choose $\mathcal J$ as the class of all injective sheaves (see \cite[Lemma II 3.1]{Har1}). 
\begin{prop}\label{bifun}
Under the assumptions from above let $A^\bullet\in \D^+(\mathcal A)$ and $B^\bullet\in \D^+(\mathcal B)$ be complexes such that 
$R^qF(A^i,B^j)=0$  for all $q\neq 0$ and all pairs $i,j\in \Z$. Then we have $RF(A^\bullet,B^\bullet)\simeq F^\bullet(A^\bullet,B^\bullet)$.  
\end{prop}
\begin{proof}
We show that we can enlarge the $K^+(F)(A^\bullet,\_)$-adapted subcategory $K^+(\mathcal J)$ to the $K^+(F)(A^\bullet,\_)$-adapted subcategory $K^+(\mathcal J')$ consisting of all complexes $B^\bullet$ with the property as above, i.e. $\mathcal J'$ is the full subcategory of all objects $B\in \mathcal B$ which are $F(A^i,\_)$-acyclic for every $i\in \Z$. The subcategory  
$\mathcal J'$ is $F(A^i,\_)$-adapted for every $i\in \Z$ (KS lemma 13.3.12). Thus for every  acyclic complex $B^\bullet\in K^+(\mathcal J')$ the double complex 
$F(A^\bullet,B^\bullet)$ has exact columns. Using the spectral sequence 
\[E_2^{i,j}=\mathcal H_I^i(\mathcal H_{II}^j(F(A^\bullet,B^\bullet)))\Longrightarrow \mathcal H^n(\tot(F(A^\bullet,B^\bullet)))=\mathcal H^n(K^+(F)(A^\bullet,B^\bullet))\]
we see that  $K^+(F)(A^\bullet,B^\bullet)$ is again acyclic. Hence the category $K^+(\mathcal J')$ is indeed adapted to the functor $K^+(F)(A^\bullet,\_)$ and we can use it to compute the derived functor. We get for $B^\bullet\in K^+(\mathcal J')$ as desired
\[RF(A^\bullet,B^\bullet)=R(K^+(F)(A^\bullet,\_))(B^\bullet)=F^\bullet(A^\bullet, B^\bullet)\,.\]
\end{proof}
Clearly there is an analogous statement for bifunctors which are right exact in each variable. For a fixed object $A^\bullet\in \D^+(\mathcal A)$ and $F$ as above, $G:=K^+(F)(A^\bullet,\_)$ and $\mathcal J'$ as in the proof we call every $B\in \mathcal J'$ a \textit{$G$-acyclic} object.     
\subsection{Danila's lemma and corollaries}\label{Dani}
Let $G$ be a finite group acting transitively on a finite set $I$, $R$ a ring and $M$ an $R[G]$-module admitting a decomposition $M=\oplus_{i\in I} M_i$ such that for any $i\in I$ and $g\in G$ the action of $g$ on $M$ restricted to $M_i$ is an isomorphism $g\colon M_i\overset\cong\to M_{g(i)}$. Then the $G$-action on $M$ induces a $\Stab_G(i)$-action on $M_i$, which makes the projection $M\to M_i$ a $\Stab_G(i)$-equivariant map.
\begin{lemma}[\cite{Dan}]\label{Dan}
For all $i\in I$ the projection $M\to M_i$ induces an isomorphism $M^G\overset\cong\to M_i^{\Stab_G(i)}$.
\end{lemma}
\begin{proof}
The inverse is given by 
$m_i\mapsto \oplus_{[g]\in G/\Stab_G(i)} g\cdot m_i$ with $g\cdot m_i\in M_{g(i)}$.
\end{proof}
We can globalise Danila's lemma to $G$-sheaves. Let $G$ and $I$ be as above and $G$ act on a scheme $X$. Let 
$\M$ be a $G$-sheaf on $X$ admitting a decomposition $\M=\oplus_{i\in I} \M_i$ such that for any $i\in I$ and $g\in G$ the linearization $\lambda$ restricted to $\M_i$ is an isomorphism $\lambda_g\colon \M_i\overset\cong\to g^*\M_{g(i)}$. Then the $G$-linearization of $\M$ restricts to a  $\Stab_G(i)$-linearization of $\M_i$, which makes the projection $\M\to \M_i$ a $\Stab_G(i)$-equivariant morphism. By lemma \ref{inf} for every $i\in I$ we have
\[\M\cong_G\Inf_{\Stab(i)}^G\M_i\,.\]
\begin{cor} 
Let $\pi\colon X\to Y$ be a $G$-invariant morphism of schemes.
Then for all $i\in I$ the projection $\M\to\M_i$ induces an isomorphism $(\pi_*\M)^G\overset\cong\to (\pi_*\M_i)^{\Stab_G(i)}$.
\end{cor}
\begin{proof}
For an affine open $U\subset Y$ we set $R=\reg(U)$, $M=\glob(\pi^{-1}U, \M)$ and $M_i=\glob(\pi^{-1}U, \M_i)$. Then the lemma applies and gives the  isomorphism of sheaves over $U$. Since for varying $U$ all the isomorphisms are induced by the projection to the $i$-th summand, they glue together.
\end{proof}
\begin{remark}
The assertion of the lemma respectively the corollary remains true if we consider complexes of $R[G]$-modules $M^\bullet$ and $M_i^\bullet$ respectively complexes of $G$-sheaves. Let $k$ be a field of characteristic zero. If $R$ is a $k$-algebra respectively in the case of $G$-sheaves if $X$ and $Y$ are $k$-schemes and $\pi_*$ is exact, taking invariants is exact. Hence in this case we also have Danila's Lemma for the cohomology of the complexes, i.e.
\[[\mathcal H^k(M^\bullet)]^G\cong [\mathcal H^k (M_i^\bullet)]^{\Stab_G(i)}\,.\]
\end{remark}
Let $G$ act on a scheme $X$ and let $\Eb\in \D^b_G(X)$.
Let $F$ be one of the functors $\Eb\otimes\_$, $\sHom(\Eb,\_)$ or $\sHom(\_,\Eb)$. We assume that there are enough locally free sheaves on $X$. Then the functor $F$ can be derived and we denote its right respectively left derived by $DF$. Let $H\le G$ be a subgroup. We can consider $F$ and $DF$ as functors on the $H$-equivariant categories by replacing $\Eb$ by $\Res_G^H \Eb$. 
\begin{lemma}
For $\Fb\in \D^b_H(X)$ there is in $\D_G(\QCoh(X))$ a natural isomorphism \[DF(\Inf_H^G\Fb)\simeq \Inf_H^G(DF(\Fb))\,.\]
\end{lemma}
\begin{proof}
Let $\Ab$ be a $F$-acyclic $H$-equivariant resolution of $\Fb$. Then $DF(\Fb)\simeq F(\Ab)$ in $\D_{H}(\QCoh(X))$. Since the inflation functor is exact, $\Inf(\Ab)$ is a $G$-equivariant resolution of $\Inf\Fb$. The objects $\Inf(\mathcal A^i)$ are still $F$-acyclic because of the compatibility of $DF$ with the forgetful functor and with direct sums. Thus 
$DF(\Inf\Fb)\simeq F(\Inf(\Ab))\simeq\Inf(F(\Ab))$ holds using lemma \ref{inffun}.  
\end{proof}
\begin{cor}\label{deriveddanila}
Let $G$, $I$, $X$, $\pi\colon X\to Y$, $\M$ and $F$ be as above such that $X$ and $Y$ are schemes over a field of characteristic zero and $\pi_*$ is exact. Then there are natural isomorphisms $[\pi_*DF(\M)]^G\simeq [\pi_*DF(\M_i)]^{\Stab(i)}$ and $[\pi_*D^k F(\M)]^G\cong [\pi_*D^kF(\M_i)]^{\Stab(i)}$ for every $k\in \Z$.  
\end{cor}
\begin{proof}
Using the identification of $\M$ with the inflation of $\M_i$ we can conclude by the previous lemma and the previous remark. 
\end{proof}
We can get analogous results with $F$ being the push-forward or the pull-back along an equivariant morphism.
\begin{remark}\label{morphismdanila}
 Let $G$, $I$ and $M$ be as above, $N=\oplus_{j\in J}N_j$ a further $R[G]$ module such that $G$ acts transitive on $J$ and such that $g\colon N_j\to N_{g(j)}$
for all $j\in J$. Let $\phi\colon M\to N$ be a morphism of $R[G]$-modules with components  $\phi(i,j)\colon M_i\to N_j$. Then for fixed $i\in I$ and $j\in J$ the map $\phi^G$ under the isomorphisms $M^G\cong M_i^{\Stab(i)}$ and $N^G\cong N_j^{\Stab(j)}$ of lemma \ref{Dan} is given by (see also \cite[Appendix B]{Sca1}) 
\[\phi^G\colon M_i^{\Stab(i)}\to N_j^{\Stab(j)}\quad,\quad m\mapsto\sum_{[g]\in \Stab(i)\setminus G} \phi(g(i),j)(g\cdot m)\,.\]
Clearly, there is the analogous formula in the case of $G$-sheaves.   
\end{remark}
\begin{remark}\label{notrans}
Danila's lemma and the corollaries can also be used to simplify the computation of invariants if $G$ does not act  transitively on $I$. In that case let $I_1,\dots,I_k$ be the $G$-orbits in $I$. Then $G$ acts transitively on $I_\ell$ for every $1\le\ell\le k$ and the lemma can be applied to every $M_{I_\ell}=\oplus_{i\in I_\ell}M_i$ instead of $M$. Choosing representatives $i_\ell\in I_\ell$ yields
\[M^G\cong \bigoplus_{\ell=1}^k M_{i_\ell}^{\Stab_G(i_\ell)}\,.\]   
\end{remark}
\begin{lemma}\label{tensinv}
 Let $k$ be a field of characteristic zero, $R$ a $k$-algebra, $G$ a finite group and $M$ an $R[G]$-module. Let $N$ be a $R$-module, i.e. a $R[G]$-module where $G$ is acting trivially. Then
\[(M\otimes_R^LN)^G=M^G\otimes^L_R N\,.\]
\end{lemma}
\begin{proof}
 See \cite[Lemma 1.7.1]{Sca1}.
\end{proof}
Also this lemma can be globalised to get an analogous result for $G$-sheaves.
\subsection{Pull-back along regular embeddings}\label{regemb}
Let $G$ be a finite group. In this subsection every variety is supposed to be a $G$-variety (over a fixed field $k$) and every subvariety is $G$-invariant and considered with the restricted $G$-action, i.e. all embeddings are $G$-equivariant. Also, all sheaves and the considered functors and derived functors are  equivariant. Of course one can apply the results to the non-equivariant case by setting $G=1$. Throughout the whole section $X$ is a non-singular variety.
\begin{lemma}[{\cite[Lemma A.2]{Sca1}}]
Let $A$ be a regular notherian local ring, $M_1,\dots,M_\ell$ finite Cohen-Macauley modules over $A$, such that 
\[
\codim(M_1\otimes\dots\otimes M_\ell)=\sum_{i=1}^\ell \codim(M_i)\,.
\]
Then all the higher torsion modules vanish, i.e. $\Tor^A_i(M_1,\dots,M_\ell)=0$ for all $i>0$.  
\end{lemma}
\begin{cor}\label{propint}
Let $i\colon Y\hookrightarrow X$ and $j\colon Z\hookrightarrow X$ be embeddings of Cohen-Macaulay subvarieties which intersect properly in $X$, i.e. \[\codim(Y\cap Z)= \codim(Y)+\codim(Z)\,,\] and $F$ a Cohen-Macaulay sheaf on $Z$. Then all the derived pull-backs $L^{-q}i^*j_*F$ for $q>0$ vanish.  
\end{cor}
\begin{proof}
 By projection formula $L^{-q}i^*j_*F\cong \sTor_{q}^{\reg_X}(i_*\reg_Y,j_*F)$. 
Since $\reg_Y$ and $F$ are Cohen-Macaulay, $i_*\reg_Y$ and $j_*F$ are also (see \cite[IV B prop 11]{Ser}) .
The stalks of the higher torsion sheaves can be computed as the torsion modules of the stalks so the result follows from the previous lemma. 
\end{proof}
\begin{lemma}
\begin{enumerate}\label{exactcoh} 
\item Let $f\colon X\to Y$ be a $G$-equivariant morphism such that the derived pull-back exists and $\Eb\in \D^-_G(Y)$ a complex such that the cohomology $\mathcal H^q(\Eb)$ is $f^*$-acyclic for all  $q\in \Z$.
Then $L^qf^*\Eb\cong f^*\mathcal H^{q}(\Eb)$ for all  $q\in \Z$.   
\item Let $\F\in \Coh_G(X)$ such that $\mathcal H^q(\Eb)$ is $(\F\otimes\_)$-acyclic for all integers $q$. Then $\sTor_{q}(\F,\Eb)\cong \F\otimes\mathcal H^{-q}(\Eb)$ for all $q\in \Z$
\end{enumerate}
\end{lemma}
\begin{proof}
We consider the spectral sequence
\[E^{p,q}_2=L^pf^*(\mathcal H ^q(\Eb))\Longrightarrow E^n=L^nf^*(\Eb)\]
(see \cite[p.81]{Huy}). By the assumption this spectral sequence is concentrated on the $q$-axis, hence $E^q=E^{0,q}_2$ for each integer $q$.
For the second part we use the spectral sequence
\[E^{p,q}_2=\sTor_{-p}(\F,\mathcal H ^q(\Eb))\Longrightarrow E^n=\sTor_{-n}(\F,\Eb)\,.\]
\end{proof}
\begin{lemma}
 Let $j\colon Z\to X$ be an embedding of varieties. Then  $j_*\colon \Coh_G(Z)\to \Coh_G(X)$ is a fully faithful functor.
\end{lemma}
\begin{proof}
 As in the non-equivariant case for $E,F\in \Coh_G(X)$ there are the natural isomorphisms
\[G\Hom_X(j_*E,j_*F)\cong G\Hom_Z(j^*j_* E,F)\cong G\Hom_Z(E,F)\,.\]
\end{proof}
Let $S$ be a variety and $Y,Z$ closed subvarieties. We denote the inclusions by 
\[
\begin{CD}
Y\cap Z
@>{d}>>
Y \\
@V{c}VV
@VV{b}V \\
Z
@>>{a}>
S\,.
\end{CD}
\]
\begin{lemma}\label{closedbase}
For every $F\in \QCoh_G(Z)$ the base change formula $b^*a_*F\cong d_*c^*F$ holds.
\end{lemma}
\begin{proof}
We reduce to the case $S=\Spec A$ affine and notice that $\frac M{I_Y\cdot M}=\frac M{(I_Y+I_Z)\cdot M}$ holds for every $A/I_Z$-module $M$. 
\end{proof}
\begin{lemma}\label{selfint}
 Let $j\colon S\hookrightarrow X$ be a regular embedding of codimension $c$ with vanishing ideal $\I_S$ and $F\in\Coh_G(X)$. Then for every $q\in \Z$
\[L^{-q}j^*j_*F\cong F\otimes\wedge^{q}(j^*\I_S)\,.\]
\end{lemma}
\begin{proof}
We first proof the special case $F=\reg_S$. By \ref{Gro} (iii)
\[j_*(Lj^*j_*\reg_S\otimes \wedge^cj^*\I_S^\vee[-c])\simeq j_*j^!j_*\reg_S\simeq j_*R\sHom_S(\reg_S,j^!j_*\reg_S)\simeq R\sHom_X(j_*\reg_S,j_*\reg_S)\,.\]
Now taking the $(c-q)$-th cohomology on both sides by \ref{Gro} (ii) yields
\[j_*(L^{-q}j^*j_*\reg_S\otimes \wedge^cj^*\I_S^\vee)\cong j_*(\wedge^{c-q}j^*\I_S^\vee)\,.\]
Using the fact that $j_*$ is fully faithful we can chancel it from the isomorphism. Tensoring with $\wedge^cj^*\I_S$ gives the result.
For general $F$ using the projection formula twice we get
\[j_*Lj^*j_*F\simeq j_*(Lj^*j_*F\otimes \reg_S)\simeq j_*F\otimes^L j_*\reg_S\simeq j_*(F\otimes^L Lj^*j_*\reg_S)\,.\]
We have already proven, that the $L^qj^*j_*\reg_S$ are locally free, hence $(F\otimes\_)$-acyclic. Thus we can use \ref{exactcoh} (ii) with $\Eb=Lj^*j_*\reg_S$ which proves the  general formula.
\end{proof}
Let $j\colon S'\hookrightarrow X$ and $S\hookrightarrow X$ be regular embeddings. Then the composition $j'\colon S'\hookrightarrow X$ is also a regular embedding (see e.g. \cite[Appendix B.7]{Ful}). We denote by $\pi\colon j_*\reg_{S}\to j'_*\reg_{S'}$ the natural surjection.
\begin{lemma}\label{indu}
The induced morphism of the $(-q)$-th derived pull-back functors  
\[L^{-q}j'_*(\pi)\colon L^{-q}j'^*j_*\reg_S\cong \wedge^q (j^*\I_S)_{|S'}\to \wedge^q j'^*\I_{S'}\cong L^{-q}j'^*j'_*\reg_{S'}\]
coincides with the one induced by the inclusion $\I_S\subset \I_{S'}$. 
\end{lemma}
\begin{proof}
Because of the compatibility via the forgetful functor of the equivariant derived functors with the non-equivariant ones, we can compute 
$L^{-q}j'^*(\pi)$ on the non-equivariant derived functor. The question is local. Thus we can assume that both vanishing ideals are globally generated by regular sequences
\[\I_S=(f_1,\dots, f_\alpha)\quad,\quad\I_{S'/S}=(f_{\alpha+1},\dots, f_{\alpha+\beta})\,.
\]
Now the non-equivariant derived pull-backs are computed (see e.g. \cite[chapter 11.1]{Huy}) using the Koszul complexes  $K^\bullet(f_1,\dots f_{\alpha})$ and $K^\bullet(f_1,\dots f_{\alpha+\beta})$ as free resolutions of $j_*\reg_S$ and $j'_*\reg_{S'}$. After pulling back the Koszul complexes along $j'$ all the differentials vanish
which yields the isomorphism
\[L^{-q}j'^*J_*\reg_S\cong \mathcal H^{-q}(j'^*K^\bullet(f_i,\dots,f_\alpha))\overset \cong\to \wedge^{q} j'^*\I_S\,,\,\quad e_{i_1}\wedge\dots\wedge e_{i_q}
 \mapsto f_{i_1}\wedge\dots\wedge f_{i_q}
\]
and the analogous isomorphism for $S'$ instead of $S$.
The morphism $\pi$ can be continued on the Koszul resolutions by
\[K^{-q}(f_1,\dots,f_\alpha)\to K^{-q}(f_1,\dots,f_{\alpha+\beta})\,,\, e_{i_1}\wedge\dots\wedge e_{i_q}
 \mapsto e_{i_1}\wedge\dots\wedge e_{i_q}\]
which yields the result.    
\end{proof}
The two lemmas \ref{propint} and \ref{selfint} can be combined as follows. Let again $Y$ and $Z$ be closed Cohen-Macaulay subvarieties of $X$. We assume that there is a non-singular subvariety $S\hookrightarrow X$ such that $S$ contains $Y$ and $Z$ such that $Y$ and $Z$ intersect properly in $S$. 
Note that the embedding $S\hookrightarrow X$ is regular since both $S$ and $X$ are non-singular.
We use the following notations for the closed embeddings:
\[\xymatrix{
            & Y \ar^i[dr]\ar^b[d]&  \\
   Y\cap Z \ar^d[ur]\ar_c[dr]   & S\ar^\iota[r] &  X\,.    \\
        &   Z\ar^a[u]\ar_j[ur]    & 
} \]
\begin{lemma}\label{mixedint}
Let $F\in\Coh_G(Z)$ be Cohen-Macauley. With the notations introduced above, for $q\in \Z$ there is the following isomorphism
\[L^{-q}i^*j_*F\cong d_*c^*\left(F\otimes(a^*\wedge^q \iota^*\I_S)\right)\,.\]
\end{lemma}
\begin{proof}
 We have $L^{-q}i^*j_*F\cong L^{-q}b^*(L\iota^*\iota_*(a_*F))$. Also by lemma \ref{selfint} and the projection formula
\[L^{-q}\iota^*\iota_*(a_*F)\cong a_*F\otimes\wedge^{q} (\iota^*\I_S)\cong a_*(F\otimes a^*\wedge^q N_S^\vee)\,.\] 
Since $F$ is Cohen-Macaulay and $\wedge^{q} N_S^\vee$ is locally free, the whole $F\otimes a^*\wedge^q N_S^\vee$ is Cohen-Macaulay for every $q$.
So by \ref{propint} the assumptions of \ref{exactcoh} (i) are satisfied with $f=b$ and $\Eb=L\iota^*\iota_*(a_*F)$. Thus,
\begin{align*} 
 L^{-q}i^*j_*F\cong b^*(a_*F\otimes \wedge^{q} (\iota^*\I_S))&\overset{\text{PF}}\cong b^*a_*(F\otimes a^*\wedge^q \iota^*\I_S)
\overset{\ref{closedbase}}\cong d_*c^*\left(F\otimes(a^*\wedge^q \iota^*\I_S)\right)\,.
\end{align*}
\end{proof}
Let $Y,Z$ and $S$ be as above and $S'\hookrightarrow S$ another non-singular subvariety such that $S'$ contains $Y$ and intersects properly with $Z$ in $S$. We denote the intersection by $Z'=Z\cap S'$. It is again Cohen-Macaulay (see \cite[section IV B 2]{Ser}).  
Furthermore, the subvarieties $Z'$ and $Y$ intersect properly in $S'$. Hence, the above lemma applies again and 
\[L^{-q}i^*j'_*\reg_{Z'}\cong \wedge^{q} (N_{S'}^\vee)_{|Y\cap Z}\,,\]
where $j'$ is the inclusion of $Z'$ in $X$.
\begin{lemma}\label{induced}
Let $\pi\colon j_*\reg_Z\to j'_*\reg_{Z'}$ be the natural surjection. Then the morphism 
\[L^{-q}i^*(\pi)\colon L^{-q}i^*j_*\reg_Z\cong \wedge^q (N^\vee_S)_{|Y\cap Z}\to \wedge^q (N^\vee_{S'})_{|Y\cap Z'}\cong L^{-q}i^*j'_*\reg_{Z'}\]
is the one induced by the inclusion $\I_S\subset \I_{S'}$. 
\end{lemma}
\begin{proof}
We have $\reg_{Z'}\cong \reg_Z\otimes_{\reg_S}\reg_{S'}$ and the surjection $\pi\colon \reg_Z\to \reg_{Z'}$ is induced by the surjection $\reg_S\to\reg_{S'}$. Thus we can use lemma \ref{indu}.
\end{proof}
\subsection{Partial diagonals and the standard representation}\label{diagsta}
Let $X$ be a smooth variety of dimension $d$ over $\C$. Let $I=\{i_1<\dots<i_\ell\}\subset \n$ with $\#I=\ell\ge2$. Then the \textit{partial diagonal} $\Delta_I$ is defined as the reduced subvariety given by
\[\Delta_I=\bigl\{ (x_1,\dots,x_n)\in X^n\mid x_i=x_j \,\forall\, i,j\in I\bigr\}\,.\] 
It is a $\Stab_{\sym_n}(I)=\sym_I\times\sym_{\bar I}$-equivariant smooth subvariety of codimension $d(\ell-1)$ in $X^n$. 
We denote the closed embedding by $\iota_I\colon \Delta_I\to X^n$.
The subgroup $\sym_I$ acts trivially on $\Delta_I$. Thus, the $\sym_I$-linearization of the conormal bundle $N_I^\vee$ is just a $\sym_I$-action (see subsection \ref{triv}). Since $N_I^\vee$ is locally free, this action is determined by the action on the fibers $N_I^\vee(Q)=\I_I(Q)$ over closed points $Q\in \Delta_I$.
 For $\ell\in \N$ the \textit{regular representation} $V_I$ of $\sym_I$ is the vector space $\C^I$ with $\sym_\ell$ permuting the vectors of the standard base $(e_i\mid i\in I)$. The natural representation decomposes into the direct sum of its invariants, which form the one-dimensional vector space spanned by $\sum_{i\in I}e_i$, and the space 
$\{v\in V\mid \sum_{i\in I} v_i=0\}$ spanned by the vectors $e_{i_1}-e_{i_2},e_{i_2}-e_{i_3},\dots, e_{i_{\ell-1}}-e_{i_\ell}$. The latter summand is an irreducible representation of $\sym_I$ called the \textit{standard representation} and denoted  by $\rho_I$.
For $I=[\ell]$ we denote the standart representation of $\sym_I=\sym_\ell$ also by $\rho_\ell$ or simply $\rho$.
We need the following result from local algebra (see \cite[IV Proposition 22]{Ser}).
\begin{prop}\label{calg}
 If $\{x_1,\dots,x_p\}$ are $p$ elements of the maximal ideal $\m$ of a regular local ring $A$, the following three properties are equivalent:
\begin{enumerate}
 \item $x_1,\dots,x_p$ is part of a regular system of parameters of $A$, i.e. of a system of parameters which generates $\m$.
 \item The images of $x_1,\dots,x_p$ in $\m/\m^2$ are linearly independent over $k$.
 \item The local ring $A/(x_1,\dots, x_p)$ is regular and has dimension $\dim A -p$.\\
(In particular, $(x_1,\dots, x_p)$ is a prime ideal.)
\end{enumerate}
\end{prop}
\begin{lemma}\label{regrep}
Let $I\subset \n$ and $Q\in \Delta_I$ be a closed point. Considering $\I_I(Q)$ with the natural action of $\sym_I$ (see proposition \ref{Gro}) we have $\I_I(Q)\cong \rho_I^{\oplus d}$.  
\end{lemma}
\begin{proof}
 For each closed point $P\in X$ we choose an open affine neighbourhood with coordinate ring $A(P)$. We denote the maximal ideal of $A(P)$ corresponding to $P$ by $\m(P)$, hence $\reg_{X,P}=A(P)_{\m(P)}$. Furthermore we choose a regular system of parameters $x_1(P),\dots,x_d(P)$ of $\reg_{X,P}$. By multiplying them by their denominator if necessary, we may assume that $x_1(P),\dots,x_d(P)\in\m(P)\subset A(P)$. Let $Q=(P^1,\dots,P^n)\in X^n$. Then there is an affine open neighbourhood of $Q$ in the product
$X^n$ with coordinate ring $B(Q)=A(P^1)\otimes\dots\otimes A(P^n)$. The point $Q$ corresponds to the maximal ideal 
\[\nax(Q)=\m(P^1)\otimes A\otimes \dots \otimes A + A\otimes \m(P^2)\otimes \dots \otimes A + \dots + A\otimes A\otimes \dots \otimes \m(P^n)\,.\]
Since the $(x_i(P))_{i=1,\dots,d}$ generate $\m(P)$, the family $(x_i^j(Q))_{i=1,\dots,d }^{ j=1,\dots, n}$ with
\[x_i^j(Q)=1\otimes \dots\otimes 1\otimes x_i(P^j)\otimes 1\otimes\dots\otimes 1\]
generates $\nax(Q)$. As $\reg_{X^n,Q}$ is a regular local ring of dimension $d\cdot n$, this family is a regular system of parameters.
Without loss of generality we may assume that $I=[\ell]$ and consider a point $Q=(P,\dots,P,P^{\ell+1},\dots,P^n)$ of $\Delta_{[\ell]}$.       
For $1\le i\le d$ and $1\le j\le \ell-1$, we set $\zeta_i^j=x_i^j(Q)- x_i^{j+1}(Q)$. Clearly, the $\zeta_i^j$ are elements of the stalk of the vanishing ideal $(\I_I)_Q\subset \nax(Q)\reg_{X^n,Q}$. We denote their images 
in $\nax(Q)/\nax(Q)^2$ by $\bar\zeta_i^j=\bar x_i^j(Q)-\bar x_i^{j+1}(Q)$. By \ref{calg} the vectors
$\bar x_1^1(Q),\dots ,\bar x_1^n(Q), \bar x_2^1(Q),\dots, \bar x_d^n(Q)$ form a basis of $\nax(Q)/\nax(Q)^2$.
Hence, $\bar\zeta_1^1,\dots\bar\zeta_d^{\ell-1}$ are $d(\ell-1)$ linearly independent elements in $\nax(Q)/\nax(Q)^2$. Thus, by proposition \ref{calg} the $\zeta_i^j$ generate a prime ideal in $\reg_{X^n,Q}$ of height $d(\ell-1)$. It is contained in $(\I_I)_Q$ which is also prime of the same height.
Thus, \[(\I_I)_Q=(\{\zeta_i^j\mid 1\le i\le d, 1\le j\le \ell-1\})\] and the fiber $\I_I(Q)$ has $(\bar\zeta_i^j)_{i=1,\dots, d}^{j=1,\dots,\ell-1}$
as a base. Now the isomorphism $\I_I(Q)\overset\cong\to \rho^{\oplus d}$ is given by $\bar\zeta_i^j\mapsto (0, \dots, 0, e_j-e_{j+1},0,\dots,0)$. 
\end{proof}
By proposition \ref{Gro} the derived dual of the structural sheaf of $\Delta_I$ in $X^n$ is cohomologically concentrated in degree $d(\ell-1)$ with
\[\mathcal H^{d(\ell-1)}((\iota_*\reg_{\Delta_I})^{\dv})\cong \iota_{I*}({\wedge}^{d(\ell-1)}\iota_{I}^*\I_I)^\vee\]
equipped with the $\sym_I\times\sym_{\bar I}$-linearization which is induced by the natural linearization of $\I_I$.
\begin{lemma}\label{topwedge} 
The group $\sym_I$ acts on $\mathcal H ^{d(\ell-1)}(\reg_{\Delta_I}^{\dv})$ trivially if $d$ is even and alternating if $d$ is odd.
\end{lemma}
\begin{proof}
The sheaf $({\wedge}^{d(\ell-1)}\iota_{I}^*\I_I)^\vee$ is a line bundle on $\Delta_I$. Thus, the action of $\sym_I$ is determined by the action on the fibers, which are isomorphic to the representation $\wedge^{d(\ell-1)}(\rho^{\oplus d})$ (dualizing can be dropped for one-dimensional representations). It suffices to consider the action of transpositions of neighbours $\tau=(k\,,\, k+1)$ since $\sym_I$ is generated by those permutations.
As one can compute, each such $\tau$ acts as a matrix $\rho(\tau)\in \GL(\C^{\ell-1})$ with determinant $-1$. The action of $\tau$ on $\wedge^{d(\ell-1)}(\rho^{\oplus d})$ 
is given by the determinant of $\rho(\tau)^{\oplus d}$. Thus the assertion follows. 
\end{proof}
In the following we will always consider the case that $X$ is a surface over $\C$, i.e $d=2$. In that case there is also a formula for the dimension of the  invariants of the lower exterior powers of $\rho^{\oplus d}$. 
\begin{lemma}\label{reginv}
 Let $\rho$ be the standard representation of the symmetric group $\sym_\ell$ for any $\ell \in \N$. 
The dimension of the space of invariants of the exterior products of the double direct sum of the standard representation is given by
\[\dim ( \wedge^p(\rho\oplus\rho))^{\sym_\ell}=\begin{cases} 1 \quad&\text{if $0\le p\le 2(\ell-1)$ is even}\\
                                                             0 \quad&\text{else.}
                                               \end{cases}
\]
\end{lemma}
\begin{proof}
See \cite[Lemma C.5.]{Sca1}.   
\end{proof}
\begin{cor}
For $p\in \N$ an even number and $I\subset [n]$ with $2(|I|-1)\ge p$ the sheaf $(\wedge^p N_I^\vee)^{\sym_I}=(\wedge^p \iota_I^*\I_I)^{\sym_I}$ is a line bundle on $\Delta_I$.  
\end{cor}
\begin{proof}
By lemma \ref{regrep} and lemma \ref{reginv} all fibers of $\wedge^p\iota_I^*\I_I^{\sym_I}$ are of rank one. This implies the assertion (see e.g. \cite[II Ex. 5.8]{Har2}). 
\end{proof}
\begin{lemma}\label{repres}
For every  $I\subset[n]$ with $\#I=k$ the set of left cosets is given by
\[\sym_n/(\sym_I\times \overline{\sym_I})=\{\sym_{I\to J}\mid J\subset\n,\, \#J=k\}\]
where
\[\sym_{I\to J}=\{\sigma\in\sym_n\mid \sigma(I)=J\}\,.\]
The right cosets are exactly the $\sym_{J\to I}$ with $\#J=k$.
 \end{lemma}
\begin{proof}
The sets $\sym_{I\to J}$ are invariant under multiplication by $\sym_I\times \overline{\sym_I}$ on the right and consist of $k!(n-k)!=\#(\sym_I\times \overline{\sym_I})$ elements. Thus they are indeed left cosets.
They are disjoint for distinct $J$ and $J'$. The  number of $J\subset \n$ with $\# J= k$ is $\binom n k$. Because of \[\binom n k\cdot k!(n-k)!=n!=\#\sym_n\] all left cosets are of this form. The proof for the right cosets is analogous. 
\end{proof}
\begin{lemma}\label{diaginduced}
 Let $X$ be a smooth surface over $\C$. For $k-2\ge p\ge0$ let  
\[t\colon L^{-2p}\iota_{[k]}^*\reg_{\Delta_{[k-1]}}\cong \wedge^{2p} \iota_{[k]}^*\I_{[k-1]}\to\wedge^{2p} \iota_{[k]}^*\I_{[k]}\cong L^{-2p}\iota_{[k]}^*\reg_{\Delta_{[k]}}\]
be the morphism induced by the natural surjection $\reg_{\Delta_{[k-1]}}\to \reg_{\Delta_{[k]}}$. We set $\sigma_k=\id_{[k]}$ and $\sigma_\ell=(\ell\,,\, \ell+1,\dots, k)$ for $\ell=1,\dots,k-1$. Then
\[T\colon [\wedge^{2p} \iota_{[k]}^*\I_{[k-1]}]^{\sym_{[k-1]}}\to [\wedge^{2p}\iota_{[k]}^*\I_{[k]}]^{\sym_{[k]}}\,,\, v\mapsto \sum_{\ell=1}^k \sigma_\ell t(v)\]
is an isomorphism.
\end{lemma}
\begin{proof}
Since both sides are line bundles on $\Delta_{[k]}$ the morphism $T$ is an isomorphism if and only if it is fiberwise. So let $P\in \Delta_{[k]}$.
We have $\sigma_\ell([k-1])=[k]\setminus\{\ell\}$ so by the last lemma the $\sigma_\ell$ form a system of representatives of $\sym_{[k]}/\sym_{[k-1]}$. 
Thus for an $\sym_{[k-1]}$-invariant $v$ the  element $\sum_{\ell=1}^k \sigma_\ell t(v)$ is $\sym_{[k]}$-invariant and $T$ is indeed well-defined. 
Since both spaces are one-dimensional by \ref{reginv}, it suffices to show that $T(P)$ is non-zero.
Note that by proposition \ref{reginv} and lemma \ref{indu} the map $t(P)$ is given by the wedge product of the natural inclusion $\rho_{[k-1]}\to \rho_{[k]}$.
Now, that $T(P)\colon (\wedge^{2p} \rho_{[k-1]})^{\sym_{[k-1]}}\to (\wedge^{2p} \rho_{[k]})$ is indeed an isomorphism follows from
\cite[Lemma C.6. (4)]{Sca1}.
\end{proof}

\section{Image of tautological sheaves under the Bridgeland-King-Reid equivalence}\label{bkr}
\subsection{The Bridgeland--King--Reid equivalence}
Let $G$ be a finite group acting on a smooth quasi-projective variety $M$ over $\C$. The quotient $\pi\colon M\to M/G$ always exists as a quasi-projective variety
but is in general singular. 
The following was introduced by Nakamura (see \cite{Nakam}) as a candidate for a resolution of the singularities of $M/G$ with good properties.
\begin{defin}
The \textit{G-Hilbert scheme} $\GHilb(M)$ parametrizes $G$-clusters on $M$, i.e. it is the scheme representing the functor
\[  T\mapsto\left\{\begin{matrix} Z\subset T\times X \text{ closed}  \\ \text{G-invariant subscheme} \end{matrix} \;\middle|\;
\begin{matrix}  \text{$Z$ is flat and finite over $T$,}  \\ \text{$H^0(Z_t)\cong \C^G$ as $G$-representations $\forall t\in T$} \end{matrix}\right\} \,.\]
The \textit{Nakamura-$G$-Hilbert scheme} $\Hilb^G(M)$ is defined to be the irreducible component of $\GHilb(M)$ which contains all the $\C$-valued points corresponding to free orbits. It is equipped with a universal family $\mathcal Z\subset \Hilb^G(M)\times M$ and the \textit{$G$-Hilbert-Chow morphism} \[\tau \colon \Hilb^G(M)\to M/G\] sending a $G$-cluster to the $G$-orbit supporting it. In summary there is the commutative diagram
\[\begin{CD}
\mathcal Z
@>{p}>>
X^n \\
@V{q}VV
@VV{\pi}V \\
\Hilb^G(M)
@>>{\tau}>
S^nX
\end{CD}\,. \]
\end{defin}
\begin{theorem}[\cite{BKR}]
Let the quotient $M/G$ be Gorenstein, i.e $\omega_{M/G}$ exists as a line bundle, and \[\dim(\Hilb^G(M)\times_{M/G}\Hilb^G(M))\le \dim (M)+1\,.\] Then $\tau$ is a crepant resolution of the singularities of $M/G$, i.e $\tau^*\omega_{M/G}\cong \omega_{\Hilb^G(M)}$, and the equivariant Fourier-Mukai transform
\[\Phi_{\reg_{\mathcal Z}}=Rp_*\circ q^*\colon \D^b(\Hilb^G(M))\to \D^b_G(M)\]
is an equivalence of triangulated categories.      
\end{theorem}
The functor $\Phi_{\reg_{\mathcal Z}}$ is called \textit{the Bridgeland-King-Reid equivalence.}
\begin{remark}
The group $G$ acts trivially on $\Hilb^G(M)$ and on $\mathcal Z\subset \Hilb^G(M)\times M$ by the action induced by the action on $M$.
To be formally correct the equivariant Fourier-Mukai transform $\Phi_{\reg_{\mathcal Z}}$ is the functor
$Rp_*\circ q^*\colon \D^b_G(\Hilb^G(M))\to \D^b_G(M)$ and the Bridgeland-King-Reid equivalence is the functor $\Phi_{\reg_{\mathcal Z}}\circ\triv$.
\end{remark}
\subsection{The Hilbert scheme of points on a surface}
Let $X$ be a smooth quasi-projective surface over $\C$ and $P\in \Q[x]$. Then the \textit{Hilbert functor} 
 \[  T\mapsto\left\{\begin{matrix} Z\subset T\times X  \\ \text{closed subscheme} \end{matrix} \;\middle|\;
\begin{matrix}  \text{$Z$ is flat and proper over $T$,}  \\ \text{$Z_t$ has Hilbert polynomial $P\,\, \forall t\in T$} \end{matrix}\right\} \]
is representable by a result of Grothendieck (see \cite{Gro2}). In case $P=n\in\N$ the representing scheme $X^{[n]}$ is called the \textit{Hilbert scheme of $n$ points on $X$}.
Its $\C$-valued points correspond to zero-dimensonal subschemes $\xi$ of $X$ of length \[\ell(\xi):=\dim_\C\Gamma( \xi,\mathcal O_\xi)=n\,.\]
The generic case of such subschemes are collections of $n$ distinct points. If the support of $\xi$ consists of less than $n$ points, the scheme structure has to be further specified. For example a subscheme of length 2 which is concentrated in one point is given by choosing a tangent direction of $X$ in that point. 
$X^{[n]}$ is a $2n$-dimensional quasi-projective smooth variety (see \cite{Fog}). 
It is a resolution of the symmetric product $S^nX:=X^n/\sym_n$ via the \textit{Hilbert-Chow morphism}
 \[\mu\colon X^{[n]}\to S^nX\quad,\quad \xi\mapsto\sum_{x\in\xi}\ell (\xi,x)\cdot x\,.\]
The \textit{isospectral Hilbert scheme}
is defined as $I^nX:=(X^{[n]}\times_{S^nX} X^n)_{\text{red}}$. It is a family of $\mathfrak S_n$-clusters in $X^n$ flat over $X^{[n]}$ (see \cite{Hai}).
Thus, it induces a morphism $\eta$ from $X^{[n]}$ to the moduli space $\sym_n\Hilb(X^n)$. Since the generic fiber of $I^nX$ over $X^{[n]}$ consist of 
$n$ distinct points, $\eta$ factorizes over $\Hilb^{\sym_n}(X^n)$.
\begin{theorem}[\cite{Hai}]
The morphism $\eta$ induces an isomorphism between the commutative diagrams
\[ \begin{CD}
I^nX
@>{p}>>
X^n \\
@V{q}VV
@VV{\pi}V \\
X^{[n]}
@>>{\mu}>
S^nX
\end{CD}\quad\cong\quad
 \begin{CD}
\mathcal Z
@>{p}>>
X^n \\
@V{q}VV
@VV{\pi}V \\
\Hilb^{\sym_n}(X^n)
@>>{\tau}>
S^nX
\end{CD}\,. \]
\end{theorem}
That means that $X^{[n]}$ with $I^nX$ in the role of the universal family of $\sym_n$-clusters is isomorphic to $\Hilb^{\sym_n}(X^n)$ over $S^nX$. Since $S^nX$ is Gorenstein and the Hilbert-Chow morphism is semi-small (see also \cite{Hai}), the assumptions of the Bridgeland-King-Reid theorem are fulfilled.
\begin{cor}\label{equi}
The Fourier-Mukai Transform
 \[\Phi:=\Phi^{X^{[n]}\to X^n}_{\reg_{I^nX}}\colon \D^b(X^{[n]})\to \D^b_{\mathfrak S_n}(X^n)\]
is an equivalence of triangulated categories. 
\end{cor}
The Bridgeland--King--Reid equivalence can be used to compute the extension groups of objects in the derived category of the Hilbert scheme.
\begin{cor}\label{bkrext}
 Let $\Fb,\Gb\in\D^b(X^{[n]})$. Then 
\[\Ext^i_{X^{[n]}}(\Fb,\Gb)\cong \sym_n\Ext^i_{X^n}(\Phi(\Fb),\Phi(\Gb))\quad\text{for all $i\in\Z$}\,.\] 
\end{cor}
\begin{proof}
Using the last corollary we indeed have
\[ 
\begin{aligned} \Ext^i_{X^{[n]}}(\Fb,\Gb)\cong \Hom_{\D^b(X^{[n]})}(\Fb,\Gb[i])
 &\cong \Hom_{\D^b_{\sym_n}(X^n)}(\Phi(\Fb),\Phi(\Gb[i]))\\
&\cong \Hom_{\D^b_{\sym_n}(X^n)}(\Phi(\Fb),\Phi(\Gb)[i])\\
&\cong \sym_n\Ext^i_{X^n}(\Phi(\Fb),\Phi(\Gb))\,.
\end{aligned}
\]
\end{proof}
We will abbreviate the functor $[\_]^{\sym_n}\circ\pi_*\colon \D^b_{\sym_n}(X^n)\to\D^b(S^nX)$ by $[\_]^{\sym_n}$. Note that $\pi_*$ indeed does not need to be derived since the quotient morphism $\pi$ is finite.
\begin{prop}\label{bkrstru}
In $\D^b_{\sym_n}(X^n)$ there is a natural isomorphism $\Phi(\reg_{X^{[n]}})\simeq \reg_{X^n}$. 
\end{prop}
\begin{proof}
 See \cite[Proposition 1.7.3]{Sca1}.
\end{proof}
\subsection{Tautological sheaves}
\begin{defin}
We define the \textit{tautological functor for sheaves} as
\[(\_)^{[n]}:=\pr_{X^{[n]}*}(\reg_\Xi\otimes \pr_X^*(\_))\colon\Coh(X)\to\Coh(X^{[n]})\,.\]
For a sheaf $F\in\Coh(X)$ we call its image $F^{[n]}$ under this functor the \textit{tautological sheaf associated to $F$}. 
In \cite[Proposition 2.3]{Sca2} it is shown that the functor $(\_)^{[n]}$ is exact. Thus, it induces the \textit{tautological functor for objects}
$(\_)^{[n]}\colon \D^b(X)\to \D^b(X^{[n]})$.  
For an object $F^\bullet\in\D^b(X)$ the \textit{tautological object associated to
$F^\bullet$} is $(F^\bullet)^{[n]}$.
\end{defin}
\begin{remark}\label{tautexact}
The tautological functor for objects is isomorphic to the Fourier-Mukai transform with kernel the structural sheaf of the universal family, i.e.
$(F^\bullet)^{[n]}\simeq \Phi^{X\to X^{[n]}}_{\reg_\Xi}(F^\bullet)$ for every $F^\bullet\in \D^b(X)$.
\end{remark}
\begin{remark}
If $F$ is locally free of rank $k$ the \textit{tautological bundle} $F^{[n]}$ is
locally free of rank $k\cdot n$ with fibers
$F^{[n]}([\xi])=\glob(\xi,F_{\mid \xi})$ since $\pr_{X^{[n]}}\colon \Xi\to X^{[n]}$ is flat and finite of degree $n$.
\end{remark}             
\subsection{The complex $C^\bullet$}\label{ccomplex}
\begin{defin}
For two finite totally ordered sets $M$, $L$ of the same cardinality we define $u_{M\to L}$ as the unique strictly increasing map.
Let now $N$ be totally ordered, $m\in M\subset N$ and $\sigma\in\sym_N$. We define the signs
\begin{align*}\eps_{\sigma,M}:=\sgn(u_{\sigma(M)\to M}\circ \sigma_{|M})=(-1)^{\#\{(i,j)\in M\times M\mid i<j\,,\, \sigma(i)>\sigma(j)\}} 
\end{align*}
and $\eps_{m,M}:=(-1)^{\#\{j\in M\mid j<m\}}$.
\end{defin}
To any coherent sheaf $F$ on $X$ we associate a $\sym_n$-equivariant
complex $C_F^\bullet$ of sheaves on $X^n$ as follows.
For $I\subset [n]$ with $|I|\ge 2$ we define $F_I=\iota_{I*}p_I^*F$. Here $\iota_I\colon\Delta_I\to X^n$ is the closed embedding of the partial diagonal and $p_I\colon \Delta_I\to X$ is the projection induced by the projection $\pr_i\colon X^n\to X$ for any $i\in I$. We set 
\[C^0_F=\bigoplus_{i=1}^n
p_i^*F\quad, \quad C^p_F=\bigoplus_{I\subset[n]\,,\,\mid I\mid =p+1} F_I\quad\text{for
$0< p< n$}\quad,\quad C^p_F=0 \quad\text{else}\,.\]
Let $s=(s_I)_{\mid I\mid= p+1}$ be a local section of $C^p_F$. We define the $\sym_n$-linearization of $C^p_F$ by
\[\lambda_{\sigma}(s)_I:=\eps_{\sigma,I}\cdot\sigma_*(s_{\sigma^{-1}(I)})\,,\]
where $\sigma_*$ is the flat base change isomorphism from the following diagram with $p_I\circ \sigma =p_{\sigma^{-1}(I)}$
\[\xymatrix{
\Delta_{\sigma^{-1}(I)} \ar^\sigma[r] \ar_{\iota_{\sigma^{-1}(I)}}[d]  &   \Delta_I\ar^{p_I}[r] \ar_{\iota_I}[d] &   X   \\
X^n    \ar^\sigma[r]         &         X^n  \,.   &            
}
\]
This gives also a $\sym_n$-linearization of $C^0_F$ using the convention $F_{\{i\}}:=p_i^*F$ and $\Delta_{\{i\}}:=X^n$.
Note that for $J\subset [n]$ with $|J|\ge 2$ and $i\in J$ there is by projection formula a natural isomorphism $F_J\cong F_{J\setminus \{i\}|\Delta_J}$.
Using this, we define the differentials $d^p\colon C^p_F\to C^{p+1}_F$ by the formula
\[d^p(s)_J:=\sum_{i\in J} \eps_{i,J}\cdot s_{J\setminus\{i\}\mid _{\Delta_J}}\,.\]
As one can check, $C^\bullet_F$ is indeed an $\sym_n$-equivariant complex. 
\begin{remark}\label{CF}
 Since $p_I$ is a projection and $\iota_I$ a closed embedding, the functor 
\[C^p\colon \Coh(X)\to \Coh_{\sym_n}(X^n)\quad,\quad F\mapsto C^p_F\]
is exact for all $0\le p\le n-1$. Thus, the functor 
\[C^\bullet\colon \Kom(\Coh(X))\to \Kom(\Coh_{\sym_n}(X^n))\quad F^\bullet\mapsto C^\bullet_{F^\bullet}:=\tot(K^{\bullet,\bullet})\]
is also exact, where the double complex is given by $K^{i,j}=C^i_{F^j}$. Hence, without deriving we get an exact functor $C^\bullet\colon \D^b(X)\to \D^b_{\sym_n}(X^n)$. 
\end{remark}
\begin{remark}\label{FI}
For every $I\subset\n$ the sheaf $F_I$ carries a $\Stab(I)=\sym_I\times\overline{\sym_{I}}$-linearization. As a $\sym_I\times\overline{\sym_{I}}$-sheaf we can write it as
\[F_I= \iota_{I*}(\mathfrak a_I\otimes p^*_{I} F)\,.\]
Here $\alt_I$ denotes the alternating representation of $\sym_I$, i.e. the one-dimensional representation on which $\sigma\in \sym_I$ acts by multiplication by  $\sgn(\sigma)$.
See also remark \ref{induequi} for a description of the linearization of $p_I^*F$. Note that, since the group $\sym_I$ acts trivially on $\Delta_I$, the $\sym_I$-linearization of $F_I$ is just a $\sym_I$-action. An element $\sigma\in \sym_I$ acts by multiplication by $\sgn(\sigma)$ on $F_I$.
\end{remark}
\begin{remark}\label{Cinf}  
For $1\le \ell\le n$ we fix an $I_0\subset\n$ with $\ell$ elements, e.g. $I_0=[\ell]$, and choose for any $I\subset\n$ with $\#I=\ell$ a $\sigma\in \sym_n$ with $\sigma(I)= I_0$, e.g. $\sigma={_{I_0}u_I}
\times {_{\bar I_0}u_{\bar I}}$. Then the chosen $\sigma$ form a system of representatives of $(\sym_{I_0}\times\overline{\sym_{I_0}})\setminus \sym_n$ (lemma \ref{repres}) and the canonical isomorphisms $F_{\Delta_I}\cong \sigma^*F_{\Delta_{I_0}}$ induce an isomorphism (see also lemma \ref{infrem})
\[C^{\ell-1}_F\cong \Inf_{\sym_I\times\overline{\sym_{I}}}^{\sym_n}(F_{I_0})\,.\]  
\end{remark}
\begin{lemma}\label{opencover}
 Let $X$ be a quasi-projective variety. Then $X^n$ and $S^nX$ have open coverings consisting of subsets of the form $U^n$ respectively $S^nU$ for $U\subset X$ open and affine. 
\end{lemma}
\begin{proof}
 See \cite[Lemma 1.4.3]{Sca1}.
\end{proof}
\begin{lemma}\label{invsur}
Let $F\in \Coh(X)$, $I\subset [n]$ with $|I|\ge 1$, $i\in [n]\setminus I$, and $r\colon F_I\to F_{I\cup \{i\}}$ the  morphism given by restricting local sections of $F_I$ to $\Delta_{I\cup \{i\}}$. Then the induced morphism 
\[\bar r\colon F_I^{\sym_{\bar I}}\to F_{I\cup\{i\}}^{\sym_{\overline{I\cup\{i\}}}}\]
is still surjective. 
\end{lemma}
\begin{proof}
 By the above lemma it suffices to show that the morphism $\bar r$ is surjective on the sections over $S^nU$ for every open affine subset $U\subset X$.
This can be seen by applying the lemmas 5.1.1. and 5.1.2 of \cite{Sca1} in cohomological degree zero.    
\end{proof}
\subsection{Descriptition of $\Phi(F^{[n]})$}
\begin{theorem}[\cite{Sca2}]\label{Sca}\label{Scadesc}
For every $F\in\Coh(X)$ the object $\Phi(F^{[n]})$ is cohomologically concentrated in degree zero. Furthermore, the complex $C^\bullet_F$ is a right resolution of $p_*q^*(F^{[n]})$. Hence, in $\D^b_{\sym_n}(X^n)$ there are the isomorphisms   
\[\Phi(F^{[n]})\simeq p_*q^*F^{[n]}\simeq C^\bullet_F\,.\]
\end{theorem} 
\begin{proof}
We will give a short sketch of Scala's proof here, since we need to make a statement about the naturalness of the isomorphisms in the theorem which can be read off from the proof. The object $\Phi(F^{[n]})$ is the image of $F$ under the composition of Fourier-Mukai transforms $\Phi_{\reg_{I^nX}}^{X^{[n]}\to X^n}\circ\Phi_{\reg_\Xi}^{X\to X^{[n]}}$. The composite of the two Fourier-Mukai kernels is given by the structural sheaf of the polygraph 
\[D=\{(x_1,\dots,x_n,y)\in X^n\times X\mid y\in\{x_1,\dots,x_n\}\}\subset X^n\times X\,.\]
There is a right resolution $\mathcal K^\bullet$ of $\reg_D$ defined as follows. For $i\in\n$ let  
\[D_i=\{(x_1,\dots,x_n,y)\in X^n\times X\mid y=x_i\}\subset X^n\times X\,.\]
and for $I\subset\n$ set $D_I=\cap_{i\in I} D_i$. Then $\mathcal K^p=\oplus_{\#I=p+1}\reg_{D_I}$ and the differentials are given by the restrictions $\reg_{D_I}\to\reg_{D_{I\cup\{i\}}}$ with signs as in the complex $C^\bullet_F$. The terms $\mathcal K^p$ are all flat over $X$. Since $\mathcal K^\bullet$ is a right resolution of $\reg_D$, it follows that $\reg_D$ is also flat over $X$. Furthermore $D$ is finite over $X^n$. Thus $\Phi(F^{[n]})$ is indeed concentrated in degree 0 and can be computed as
\[\Phi(F^{[n]})\simeq\Phi_{\reg_D}^{X\to X^n}(F)\simeq p_{X^n*}( p_X^*F\otimes \mathcal K^\bullet)\,.\]
Now $p_{X^n|D_I}\colon D_I\to \Delta_I$ is an isomorphism which yields $p_{X^n*}( p_X^*F\otimes \mathcal \reg_{\Delta_I})\cong F_I$ and hence the isomorphisms of the  theorem. 
\end{proof}
\begin{remark}\label{natural}
Every morphism $\phi\colon E\to F$ of coherent sheaves on $X$ induces a morphism $\phi^{[n]}=\pr_{X^{[n]}*}(\reg_\Xi\otimes \pr_X^* \phi)\colon E^{[n]}\to F^{[n]}$. Under the isomorphisms of the theorem 
\[p_*q^*(\phi^{[n]})\simeq 
 \Phi(\phi^{[n]})\simeq\Phi_{\reg_D}^{X\to X^n}(\phi)\simeq p_{X^n*}( p_X^*\phi\otimes \mathcal K^\bullet)\,.\]
corresponds to $C^\bullet (\phi)\colon C^\bullet_E\to C^\bullet_F$. Thus, we can rephrase the  theorem by saying that there is an isomorphism of functors
\[p_*q^*(\_)^{[n]}\cong C^\bullet\colon \Coh(X)\to \Kom(\Coh_{\sym_n}(X^n))\,.\]
The functors $(\_)^{[n]}$ and $q^*$ are exact ($q$ is flat). Furthermore, $q^*F^{[n]}$ is $p_*$-acyclic for every tautological sheaf $F^{[n]}$ by the above theorem. Thus for every $E^\bullet\in \D^b(X)$ we have a natural isomorphism $\Phi((E^\bullet)^{[n]})\simeq p_*q^* (E^{[n]})^\bullet$. Since $C^\bullet$ is also an exact functor, we have the following isomorphism of functors on the level of derived categories:  
\[\Phi((\_)^{[n]})\simeq C^\bullet \colon \D^b(X)\to \D^b_{\sym_n}(X^n)\,.\]
So $\Phi((F^\bullet)^{[n]})\simeq C^\bullet_{F^\bullet}$ holds for every $F^\bullet\in \D^b(X)$.  
\end{remark}
There is also a result on the image under the  Bridgeland-King-Reid equivalence of tensor products of tautological sheaves with multiple factors. We will state here the special case of tautological bundles.
\begin{theorem}[Scala, \cite{Sca2}]\label{tenstaut}
Let $E_1,\dots,E_k$ be locally free sheaves on the smooth quasi-projective surface $X$. Then $\Phi(\otimes_{i=1}^k E_i^{[n]})$ is cohomologically concentrated in degree zero and hence isomorphic in $\D^b_{\sym_n}(X^n)$ to
$p_*q^*(\otimes_{i=1}^k E_i^{[n]})$. Furthermore we have $p_*q^*(\otimes_{i=1}^k E_i^{[n]})\cong E^{0,0}_\infty$ where $E^{0,0}_\infty$ is on the limiting sheet of the spectral sequence
\[E^{p,q}_1=\bigoplus_{i_1+\dots+i_k=p}\sTor_{-q}(C^{i_1}_{E_1},\dots,C^{i_k}_{E_k})\Longrightarrow E^n=\mathcal H^n(C^\bullet_{E_1}\otimes^L\dots\otimes ^L C^\bullet_ {E_k})\,.\]
The differentials of the spectral sequence are induced by the differentials of the complexes $C^\bullet_{E_i}$.     
\end{theorem} 
\begin{remark}\label{koneinva}
Let $E_1,\dots, E_k$ be locally free and $F$ an arbitrary coherent sheaf on $X$.
By theorem \ref{Sca} we can make the identification $p_*q^*F^{[n]}=\ker(d^0_F\colon C^0_F\to C^1_F)$. Thus, we can describe the space of sections of $p_*q^*F^{[n]}$ over $U^n$ as the sections $s=(s_1,\dots,s_n)$ of $C^0_F$ with $s_{i|\Delta_{ij}}=s_{j|\Delta_{ij}}$ for every distinct $i,j\in\n$ or alternatively
\[p_*q^*F^{[n]}(U^n)=\{s\in C^0_F(U^n)\mid s_{|\Delta_{ij}} \text{ is $(i\,\,j)$-invariant } \forall i,j\in\n,\, i\neq j\}\,.\]    
The differential 
\[d_0^{0,0}\colon E_0^{0,0}=\bigotimes_{i=1}^k C_{E_i}^0\to E_0^{1,0}=\bigoplus_{\ell=1}^k \left(C_{E_\ell}^1\otimes \bigotimes_{i\in[n]\setminus\{\ell\}} C_{E_i}^0\right)\,.\]
on the 1-level of the spectral sequence of theorem \ref{tenstaut} is given by the direct sum of the maps $d_{E_\ell}^0\otimes\id_{\otimes_{i\in [n]\setminus \{\ell\}} C_{E_i}^0}$. Since the $\sym_n$-linearization of $C^0_{E_1}\otimes \dots \otimes C_{E_k}^0$ is given factor-wise, this implies that every local section of $E_1^{0,0}=\ker(d_0^{0,0})$, and thus also of $p_*q^*(\otimes_i E_i^{[n]})\cong E_\infty^{0,0}\subset E_1^{0,0}$, has the  property that its restriction to $\Delta_{ij}$ is $(i\,\,j)$-invariant for every pair $1\le i<j\le n$.  
\end{remark}
\section{Extension groups of twisted tautological objects}\label{chapfour}
Throughout the rest of the text let $X$ be a smooth quasi-projective surface over the complex numbers $\C$, $n\ge 2$ a positive integer and $X^{[n]}$ the Hilbert scheme of $n$ points on $X$. We will use the  Bridgeland-King-Reid equivalence to compute the extension groups of certain objects $\E^\bullet,\F^\bullet\in \D^b(X^{[n]})$ by the formula (see corollary \ref{bkrext})
\[\Ext^i_{X^{[n]}}(\E^\bullet,\F^\bullet)\cong \sym_n\Ext^i_{X^n}(\Phi(\E^\bullet),\Phi(\F^\bullet))\,.\]
We can rewrite the right hand side as
\[
\begin{aligned}
\sym_n\Ext^i_{X^n}(\Phi(\E^\bullet),\Phi(\F^\bullet)))
&\cong R^i\glob(S^nX,[\pi_*R\sHom_{X^n}(\Phi(\Eb),\Phi(\Fb)]^{\sym_n})\\
&\cong \Hyp^i(S^nX,[\pi_*R\sHom_{X^n}(\Phi(\Eb),\Phi(\Eb)]^{\sym_n})\,.
\end{aligned}\]
The functors $\pi_*$ and $[\_]^{\sym_n}$ are indeed exact and hence need not be derived. We will first compute the inner expression $[\pi_*R\sHom_{X^n}(\Phi(\Eb),\Phi(\Eb))]^{\sym_n}$. We abbreviate the occurring bifunctor by 
\[\ihom(\_,\_):=[\pi_*\sHom_{X^n}(\_,\_)]^{\sym_n}\]
and set $\iext^i=R^i\ihom$.
Because of the exactness of $[\_]^{\sym_n}:=[\_]^{\sym_n}\circ \pi_*$ we have
\[R\ihom(\_,\_)\simeq [R\sHom_{X^n}(\_,\_)]^{\sym_n}\,.\]
\subsection{Computation of the $\ihom$s}
\begin{lemma}\label{normalhom}
 Let $X$ be a normal variety together with $U\subset M$ an open subvariety such that $\codim(X\setminus U,X)\ge 2$. Given two locally free sheaves $F$ and $G$ on $X$ and 
a subsheaf $E\subset F$ with $E_{|U}=F_{|U}$ the maps $a\colon\Hom(F,G)\to \Hom(E,G)$ and $\hat a\colon\sHom(F,G)\to \sHom(E,G)$, given by restricting the domain of the morphisms, are isomorphism.
\end{lemma}
\begin{proof}
It suffices to show that $a$ is an isomorphism. Since every open $V\subset X$ is again a normal variety with $V\setminus(U\cap V)$ of codimension at least 2, it follows by considering an open affine covering that 
$\hat a$ is also an isomorphism.   
We construct the inverse $b\colon\Hom(E,G)\to \Hom(F,G)$ of $a$. For a morphism $\phi\colon E\to G$ the morphism $b(\phi)$ sends $s\in F(V)$ to the unique section in $G(V)$ which restricts to $\phi(s_{|V\cap U})\in G(V\cap U)$.   
\end{proof}
We denote by $\mathbbm D$ the big diagonal in $X^n$, i.e.
\[\mathbbm D=\bigcup_{1\le i<j\le n} \Delta_{ij}=\{(x_1,\dots, x_n)\mid \#\{x_1,\dots,x_n\}< n\}\subset X^n\,,\]
and by $U=X^n\setminus \mathbbm D$ its open complement in $X^n$.
\begin{prop}\label{invhom}
 Let $k\in\N$ and $E_1,\dots,E_k,F\in \Coh(X)$ be locally free sheaves. Then there are  natural isomorphisms
\begin{enumerate}
 \item $\ihom(\Phi(E_1^{[n]}\otimes\dots\otimes E_k^{[n]}),\reg_{X^n})\simeq \ihom(C^0_{E_1}\otimes \dots\otimes C^0_{E_k},\reg_{X^n})$,
\item $\ihom(\Phi(E_1^{[n]}\otimes\dots\otimes E_k^{[n]}),\Phi(F^{[n]}))\simeq \ihom(C^0_{E_1}\otimes \dots\otimes C^0_{E_k},C^0_F)$.
\end{enumerate}
\end{prop}
\begin{proof}
 By theorem \ref{tenstaut} we have $p_*q^*(E_1^{[n]}\otimes\dots\otimes E_k^{[n]})\simeq\Phi(E_1^{[n]}\otimes\dots\otimes E_k^{[n]})$. Moreover, $p_*q^*(\otimes_{i=1}^k E_i^{[n]})$ can be identified with the subsheaf $E_\infty^{0,0}$ of $E_0^{0,0}=C^0_{E_1}\otimes \dots\otimes C^0_{E_k}$. Since the terms $E_0^{p,q}$ of the spectral sequence are for $p\ge 1$ all supported on $\mathbbm D$, we have  
\[p_*q^*(E_1^{[n]}\otimes\dots\otimes E_k^{[n]})_{|U}=  (C^0_{E_1}\otimes \dots\otimes C^0_{E_k})_{|U}\,.\]
Since $X^n$ is normal and $\mathbbm D$ of codimension 2, lemma \ref{normalhom} yields the first isomorphism of the proposition (even before taking invariants). For the second one we consider an open subset $W\subset  S^nX$ and a $\sym_n$-equivariant morphism 
\[\phi\colon p_*q^*(E_1^{[n]}\otimes\dots\otimes E_k^{[n]})_{|\pi^{-1}W}\to C^0_{F|\pi^{-1}W}\,.\]
By remark \ref{koneinva} $\phi$ factorises over $p_*q^*F^{[n]}_{|\pi^{-1}W}$. This yields 
\begin{align*}\ihom( p_*q^*(E_1^{[n]}\otimes\dots\otimes E_k^{[n]}),p_*q^*F^{[n]})&\cong\ihom( p_*q^*(E_1^{[n]}\otimes\dots\otimes E_k^{[n]}),C^0_F)\\&\cong
 \ihom( C^0_{E_1}\otimes\dots\otimes C^0_{ E_k}, C^0_F)
\end{align*}
where for the second isomorphism we use again lemma \ref{normalhom}. 
\end{proof}
\subsection{Vanishing of the higher $\iext^i(\Phi(F^{[n]}),\reg_{X^n})$}
\begin{remark}\label{globadjoint}
 Let $H\le \sym_n$ be a subgroup, $E$ a $H$-sheaf and $F$ a $\sym_n$-sheaf on $X^n$. Since $X^n$ has a covering by open affine $\sym_n$-invariant subsets, namely by $U^n$ for $U\subset  X$ open and affine (lemma \ref{opencover}), the adoint property of the inflation functor globalises to a natural isomorphism 
\[[\sHom(\Inf_H^{\sym_n}E, F)]^{\sym_n}\cong [\sHom(E,F)]^H\,.\]
This also gives a formula for the derived functors $R\sHom$, namely \[[R\sHom(\Inf_H^{\sym_n}E^\bullet, F^\bullet)]^{\sym_n}\simeq [R\sHom(E^\bullet,F^\bullet)]^H\] for $E^\bullet\in \D^b_H(X^n)$ and $F^\bullet\in \D^b_{\sym_n}(X^n)$. Alternatively, we can apply lemma \ref{deriveddanila} with regard to the functor $\sHom(\_, F^\bullet)$ to get the same isomorphism.   
\end{remark}
\begin{lemma}
 Let $F$ be a locally free sheaf on $X$. Then $[(C^p_F)^\dv]^{\sym_n}\simeq 0$ for $p>0$.
\end{lemma}
\begin{proof}
Using that $C^p_F\cong\Inf_{\sym_n}^{\sym_{[p+1]}\times \sym_{\overline {[p+1]}}}F_{[p+1]}$ (remark \ref{Cinf}) and the previous remark we get $[(C^p_F)^\dv]^{\sym_n}\simeq [F_{[p+1]}^\dv]^{\sym_{[p+1]}\times \sym_{\overline {[p+1]}}}$.
By remark \ref{FI}, the equivariant Grothendieck duality for regular closed embeddings (proposition \ref{Gro}), and the fact that $F$ is locally free, we have in $\D^b_{\sym_{[p+1]}\times \sym_{\overline {[p+1]}}}(X^n)$ isomorphisms  
\begin{align*} 
(F_{[p+1]})^\dv&\simeq R\sHom_{X^n}\left(\iota_{[p+1]*}(\alt\otimes p^*_{[p+1]} F),\reg_{X^n}\right)\\
&\simeq \iota_{[p+1]*}R\sHom_{\Delta_{[p+1]}}(\alt\otimes p^*_{[p+1]} F,\iota_{[p+1]}^!\reg_{X^n})\\
&\simeq \iota_{[p+1]*}\left(\alt\otimes (p^*_{[p+1]} F)^\vee \otimes (\wedge^{2p}\iota_{[p+1]}^*\I_{[p+1]})^\vee   \right)[-2p]\,.     
\end{align*}
The group $\sym_{[p+1]}$ acts trivially on $p^*_{[p+1]} F$ and on $\wedge^{2p}\iota_{[p+1]}^*\I_I$ (lemma \ref{topwedge}). Thus, it acts alternating on the whole $\alt\otimes (p^*_{[p+1]} F)^\vee \otimes (\wedge^{2p}\iota_{[p+1]}^*\I_I)^\vee$ which makes the $\sym_{[p+1]}$-invariants vanish. Clearly, this implies the vanishing of the 
$\sym_{[p+1]}\times \sym_{\overline {[p+1]}}$-invariants.     
\end{proof}
\begin{prop}\label{Rdual}
 For every locally free sheaf $F$  on $X$ there is  in $\D^b(S^nX)$ a natural isomorphism $[(\Phi(F^{[n]}))^\dv]^{\sym_n}\simeq [(C^0_F)^\vee]^{\sym_n}$. In particular, all $\iext^i(\Phi(F^{[n]}),\reg_{X^n})$ vanish for $i>0$.
\end{prop}
\begin{proof}
By theorem \ref{Sca} in $\D^b_{\sym_n}(X^n)$ there is the isomorphism $\Phi(F^{[n]})\simeq C^\bullet_F$. The $\sym_n$-sheaf $C^0_F$ is locally free, hence $[(\_)^\vee]^{\sym_n}$-acyclic. For $p>0$ the terms $C^p_F$ are also $[(\_)^\vee]^{\sym_n}$-acyclic by the previous lemma. Thus, $[(\Phi(F^{[n]})^\dv]^{\sym_n}$ can be computed by applying the non derived functor $[(\_)^\vee]^{\sym_n}$ to $C^\bullet_F$. 
Again by the previous lemma $[(C^p_F)^\vee]^{\sym_n}=0$ for $p>0$ and  the proposition follows. 
 \end{proof}
\subsection{Vanishing of the higher $\iext^i(\Phi(E^{[n]}),\Phi(F^{[n]}))$}
\begin{remark}\label{part}
For $K\subset M\subset \n$ we denote the closed embedding of the corresponding partial diagonals by $\iota_M^K\colon \Delta_M\to \Delta_K$.   
For an subset $I\subset \n$ with $|I|\ge2$ the codimension of the partial diagonal $\Delta_I=\{(x_1,\dots,x_n)\mid x_i=x_j\,\forall i,j\in I\}$ in $X^n$ is $2(|I|-1)$. For $|I|\le 1$ we set $\Delta_I:=X^n$. Let $I$ and $J$ be subsets of $\n$. If $I\cap J=\emptyset$ the partial diagonals $\Delta_I$ and $\Delta_J$ intersect properly in $X^n=\Delta_{I\cap J}$. In this case we denote the embeddings of $\Delta_I\cap \Delta_J$ into $\Delta_I$ respectively $\Delta_J$ by $\iota_{I\cup J}^I$ respectively $\iota_{I\cup J}^J$ although the intersection does not equal $\Delta_{I\cup J}$. If $I\cap J\neq \emptyset$ we have indeed $\Delta_I\cap \Delta_J=\Delta_{I\cup J}$. Furthermore
\begin{align*}\codim(\Delta_{I\cup J},\Delta_{I\cap J})=2|(I\cup J)\setminus (I\cap J)|&=2|I\setminus (I\cap J)|+ 2| J\setminus (I\cap J)|\\&=\codim(\Delta_{I},\Delta_{I\cap J})+\codim(\Delta_{J},\Delta_{I\cap J}) \,,\end{align*}
i.e. $\Delta_I$ and $\Delta_J$ again intersect properly in $\Delta_{I\cap J}$. In summary for $I,J\subset\n$ the diagram
\[\xymatrix{
            & \Delta_I \ar^{\iota_I}[dr]\ar^{\iota_I^{I\cap J}}[d]&  \\
   \Delta_I\cap \Delta_J \ar^{\iota_{I\cup J}^I}[ur]\ar_{\iota_{I\cup J}^J}[dr]   & \Delta_{I\cap J}\ar^{\iota_{I\cap J}}[r] &  X^n\,.    \\
        &   \Delta_J\ar^{\iota_J^{I\cap J}}[u]\ar_{\iota_J}[ur]    & 
} \] 
fulfils the assumptions of lemma \ref{mixedint}. This yields for every Cohen-Macauley sheaf $\F$ on $\Delta_J$ and any $q\in \Z$ the formula
\[L^{-q}\iota_I^*\iota_{J*}\F\cong \iota_{I\cup J*}^I\iota_{I\cup J}^{J*}\left(\F\otimes(\iota_J^{I\cap J*}\wedge^q N^\vee_{I\cap J})\right)\,.\]
\end{remark}
Let $E$ and $F$ be locally free sheaves on $X$. In order to compute the invariants of the higher extension sheaves we will use the spectral sequence $A$ associated to the functor $\ihom(\_,p_*q^*F^{[n]})$ 
given by (see e.g. \cite[Remark 2.67]{Huy})
\[A_1^{p,q}=[\sExt^q(C^{-p}_E,p_*q^*F^{[n]})]^{\sym_n}\Longrightarrow A^m= [\sExt^m(C^\bullet_E,p_*q^*F^{[n]})]^{\sym_n}\cong
[\sExt^m(\Phi(E^{[n]}),\Phi(F^{[n]})]^{\sym_n}\,.\]
The terms in the  $k$-th column of $A_1$ in turn are computed by the spectral sequence $B(k)$ associated to $\ihom(C^k_E,\_)$ and given by
\[B(k)_1^{p,q}= [\sExt^q(C^{k}_E,C^p_F)]^{\sym_n}\Longrightarrow B(k)^m= [\sExt^m(C^k_E,C^\bullet_F)]^{\sym_n}\simeq
[\sExt^m(C^k_E,p_*q^*F^{[n]})]^{\sym_n}\,.\]
Here as direct summands terms of the form $\sExt^q(E_I,F_J)$ occur. For $I,J\in \n$ with $\#I= c+1$ and $\#J=d+1$ these are $\sym_{I,J}$-equivariant sheaves where
\[\sym_{I,J}:=\Stab_{\sym_n}(I,J)=(\sym_I\times \sym_{\bar I})\cap (\sym_J\times \sym_{\bar J})= \sym_{I\setminus J}\times \sym_{J\setminus I}\times \sym_{I\cap J}\times \sym_{\overline {I\cup J}}\,.\]
In $\D^b_{\sym_{I,J}}(X^n)$ we have the isomorphisms
\begin{align*}
 R\sHom_{X^n}(E_I,F_J)&\simeq E_I^\dv\otimes^L F_J\\
&\overset{\ref{Gro}}\simeq \iota_{I*}\left(p_I^*E^\vee\otimes \alt_I\otimes (\wedge^{2c}\iota_I^*\I_I)^\vee\right)\otimes^L_{X^n} \iota_{J*}(p_J^* F\otimes 
\alt_J)[-2c]\\ 
&\overset{\text{PF}}\simeq \iota_{I_*}\left(p_I^*E^\vee\otimes \alt_I\otimes (\wedge^{2c}\iota_I^*\I_I)^\vee \otimes L\iota_I^*\iota_{J*}(p_J^* F\otimes 
\alt_J)\right)[-2c]\,.
\end{align*}
We see that $R\sHom_{X^n}(E_I,F_J)$ has no cohomology in degrees greater than $2c$. 
Taking the $(2c-q)$-th cohomology for $q\ge 0$ on both sides yields
\begin{align*}
&\sExt^{2c-q}(E_I,F_J)\\
&\overset{\text{lf}}\cong 
\iota_{I*}\left( p_I^*E^\vee\otimes \alt_I\otimes (\wedge^{2c}\iota_I^*\I_I)^\vee\otimes L^{-q}\iota_I^*\iota_{J*}(p_J^* F\otimes 
\alt_J)\right) \\
&\overset{\ref{part}}\cong \iota_{I*}\left((p_I^*E\otimes \alt_I\otimes (\wedge^{2c}\iota_{I}^*\I_{I}))^\vee\otimes \iota_{I\cup J*}^I
(\wedge^{q}N^\vee_{I\cap J|\Delta_{I}\cap\Delta_J}\otimes(p_J^* F\otimes 
\alt_J)_{|\Delta_{I}\cap\Delta_J})\right)\,.\\
&\overset{\text{PF}}\cong \iota_{I\cup J*}\underbrace{\left(\iota_{I\cup J}^{I*}(p_I^*E\otimes \alt_I\otimes (\wedge^{2c}\iota_{I}^*\I_{I}))^\vee\otimes
\wedge^{q}N^\vee_{I\cap J|\Delta_{I}\cap \Delta_J}\otimes(p_J^* F\otimes 
\alt_J)_{|\Delta_{I}\cap \Delta_J}\right)}\,.
\end{align*}
We abbriviate the underbraced $\sym_{I,J}$-sheaf on $\Delta_I\cap \Delta_J$ by $T^q_{I,J}$.
We have $\sExt^{2c-q}(E_I,F_J)=0$ for $q>\codim(\Delta_{I\cap J},X^n)=2(|I\cap J|-1)$.
\begin{lemma}\label{IJvanish}
Let $I,J\subset\n$ such that $\#(I\setminus J)\ge 2$ or $\#(J\setminus I)\ge 2$. Then 
\[[R\sHom_{X^n}(E_I,F_J)]^{\sym_{I,J}}\simeq 0\,.\]
\end{lemma}
\begin{proof}
Since $\sExt^{2c-q}(E_I,F_J)$ is by the  computation above the push-forward of the sheaf $T^q_{I,J}$, which is defined on $\Delta_{I}\cap\Delta_J$, the $\sym_{I\setminus J}\times \sym_{J\setminus I}\times \sym_{I\cap J}$-linearization of it is just an ordinary group action. Assume that $\#(I\setminus J)\ge 2$ and choose a transposition 
$\tau\in \sym_{I\setminus J}$. Then $\tau$ acts by $-1$ on $\alt_I$ and trivially on all other tensor factors of $T^q_{I,J}$. Hence, it acts by $-1$ on the whole $T^q_{I,J}$ which makes the invariants vanish. The case $\#(J\setminus I)\ge 2$ is analogous.   
\end{proof}
\begin{remark}\label{oddvanish}
For $I\cap J\neq \emptyset$ the $\sym_{I\cap J}$-action on $T^q_{I,J}$ is given by the $\sym_{I\cap J}$-action on the factor $\wedge^{q}N^\vee_{I\cap J|\Delta_{I\cup J}}$ since $\sym_{I\cap J}$ acts alternating on two of the other tensor factors of $T^q_{I,J}$, namely on $\alt_I$ and $\alt_J$, and trivially on the remaining three. Thus, by lemma \ref{tensinv} the invariants are given by     
\[[T^q_{I,J}]^{\sym_{I\cap J}}=\left(p_I^*E^\vee\otimes (\wedge^{2c}N_I)\otimes
p_J^* F \right)_{|\Delta_{I\cup J}}\otimes [\wedge^{q}N^\vee_{I\cap J|\Delta_{I\cup J}}]^{\sym_{I \cap J}}\,.\]
In particular, by the lemmas \ref{regrep} and \ref{reginv} the invariants vanish for $q$ odd.
\end{remark}
\begin{remark}
Let $k,p\in[n]$. We set $\mathcal P_p= \{J\subset\n\mid \#J=p\}$. The orbits of 
$\mathcal P_p$ under the action of $\sym_{[k]}\times \sym_{\overline{[k]}}$ are exactly the sets $_k\mathcal P_p^i=\{J\in \mathcal P_p\mid \#(J\cap [k])=i\}$ for $i=0,\dots,\min\{k,p\}$. A canonical choice of representatives of the orbits is
\[_kI_p^i:=\{1,\dots,i,k+1,\dots,k+p-i\}=[i]\cup[k+1,k+p-i]\in {_k\mathcal P}_p^i\,.\]
The stabilizer of $_kI_p^i$ is given by $\sym_{[k],_kI_p^i}$ (see above). 
Similarly, a system of representatives of the orbits of $\mathcal P_p$ under the $\sym_{[k]}$-action is given by all the sets of the form $[i]\cup M$ with $i=0,\dots, \min\{k,p\}$ and $M\subset [k+1,n]$.     
\end{remark}
\begin{lemma}
Let $E$, $F$ be locally free sheaves and $B$ the spectral sequence described above. For $k=0,\dots, n-1$ the only non-vanishing term on the 2-sheet of $B(k)$  is $B(k)^{k,0}_2$.
\end{lemma}
\begin{proof}
Using remark \ref{globadjoint} and lemma \ref{deriveddanila} together with remark \ref{notrans} yields 
\begin{align*}
B(k)^{p,2k-q}_1=[\sExt^{2k-q}(C^k_E,C^p_F)]^{\sym_n}&\cong [\sExt^{2k-q}(\Inf_{\sym_{[k+1]}\times \overline{\sym_{[k+1]}}}^{\sym_n} E_{[k+1]},C^p_F)]^{\sym_n}\\
&\cong [\sExt^{2k-q}(E_{[k+1]}, C^p_F)]^{\sym_{[k+1]}\times \overline {\sym_{[k+1]}}}\tag 1\\
&\cong [\sExt^{2k-q}(E_{[k+1]},\bigoplus_{I\in \mathcal P_{p+1}} F_I)]^{\sym_{[k+1]}\times \overline {\sym_{[k+1]}}}\\
&\cong \bigoplus_{i=0}^{\min\{k,p\}}[\iota_{[k+1]\cup {_kI_p^i}*}T^q_{[k+1],_kI_p^i}]^{\sym_{[k+1],{_kI_p^i}}}\tag 2\,.
\end{align*}
By lemma \ref{IJvanish} we see that $B(k)^{p,2k-q}_1$ vanishes whenever $p\notin\{ k-1,k,k+1\}$. Furthermore, $B(k)^{p,2k-q}_1$ vanishes whenever $q$ is odd by remark \ref{oddvanish}. Thus, the only non-trivial terms on the 1-level of $B(k)$ are organised in the short sequences
\[ 0\to B(k)^{k-1,2k-q}_1\to B(k)^{k,2k-q}_1\to B(k)^{k+1,2k-q}_1\to 0\]
for $q\in[2k]$ even.  We first will show that these sequences are exact for $q<2k$, i.e. for $2k-q\ge 2$. By (1) they are isomorphic to the $(\sym_{k+1}\times\overline{\sym_{k+1}})$-invariants of the sequences
\begin{align*}0\to \sExt^{2k-q}(E_{[k+1]}, C^{k-1}_F)\to \sExt^{2k-q}(E_{[k+1]}, C^{k}_F)\to \sExt^{2k-q}(E_{[k+1]}, C^{k+1}_F)\to 0 \tag 3\,.\end{align*}
All sheaves in this sequences are push-forwards of sheaves on $\Delta_{[k+1]}$ so the $\sym_{k+1}$-linearization on them reduces to a $\sym_{k+1}$-action. We show that the sequence is already exact after taking the $\sym_{k+1}$-invariants. 
By (2) and lemma \ref{IJvanish} the $\sym_{k+1}$-invariants of the sequence (3) are given by the sequence 
\begin{align*}[T^q_{[k+1],[k]}]^{\sym_{k}}\overset \phi\to [T^q_{[k+1],[k+1]}]^{\sym_{k+1}}\oplus\bigoplus_{i=k+2}^n [T^q_{[k+1],[k]\cup \{i\}}]^{\sym_{k}}
\overset \psi\to \bigoplus_{i=k+2}^n[T^q_{[k+1],[k+1]\cup\{i\}}]^{\sym_{k+1}}\tag 4\end{align*}
where we left out the push-forwards along the closed embeddings in the notation.
We denote the components of $\phi$ and $\psi$ by 
\begin{align*}\phi^0\colon [T^q_{[k+1],[k]}]^{\sym_{k}}\to [T^q_{[k+1],[k+1]}]^{\sym_{k+1}}\,&,\,\phi^i\colon [T^q_{[k+1],[k]}]^{\sym_{k}}\to [T^q_{[k+1],[k]\cup\{i\}}]^{\sym_{k}}\\ \psi^j_0\colon [T^q_{[k+1],[k+1]}]^{\sym_{k+1}}\to [T^q_{[k+1],[k+1]\cup\{j\}}]^{\sym_{k+1}}\,&,\, 
 \psi^j_i\colon [T^q_{[k+1],[k]\cup\{i\}}]^{\sym_{k}}\to [T^q_{[k+1],[k+1]\cup\{j\}}]^{\sym_{k+1}}\,.
\end{align*}
The direct summands occurring in (4) are of the form $[T^{q}_{I,J}]^{\sym_{I\cap J}}$ with $1\in I\cap J$.
Thus, by remark \ref{oddvanish} they are given by
\[[\iota_{I\cup J*}T^{q}_{I,J}]^{\sym_{I\cap J}}=\iota_{I\cup J*}\left((p_1^*E^\vee\otimes (\wedge^{2c}N_I)\otimes p_1^*F)_{|\Delta_{I\cup J}}\otimes(\wedge^q\iota_{I\cup J}^*\I_{I\cap J})^{\sym_{I\cap J}}\right)\,.\]
 The differentials in $B(k)_1$ are induced by the differentials of the complex $C^\bullet_F$ whose components 
\[\pi_{I,i}\colon F_I\cong \iota_ {I*}\reg_{\Delta_I}\otimes \pr_{\min I}^*F\to \iota_{I\cup\{i\}*}\reg_{\Delta_{I\cup \{i\}}}\otimes \pr_{\min I}^*F\cong F_{I\cup\{i\}}\] are given by the natural surjections $\iota_ {I*}\reg_{\Delta_I}\to \iota_ {I\cup\{i\}*}\reg_{\Delta_{I\cup\{i\}}}$ up to the sign $\eps_{i,I\cup\{i\}}$. Thus, by \ref{induced} the differentials in the sequence (4) coincide up to the sign $\eps_{i,I\cup\{i\}}$ with the maps induced by the inclusion 
$\I_{J\cap I}\subset \I_{J\cap (I\cup \{i\})}$ on the factor $(\wedge^q\iota_{I\cup J}^*\I_{I\cap J})^{\sym_{I\cap J}}$.
A system of representatives of $\sym_{k+1}/\Stab_{\sym_{k+1}}([k])$ (see lemma \ref{repres}) is given by 
the $\sigma_\ell$ of lemma \ref{diaginduced}. Thus, by Danila's lemma for morphisms (remark \ref{morphismdanila}) the map $\phi^0$ coincides with the morphism $T$ from lemma \ref{diaginduced} tenzorized with the identity on $(p_1^*E^\vee\otimes \wedge^{2k}N_{[k+1]} \otimes p_1^*F)$. 
Thus, $\phi^0$ is an isomorphism.
This implies that $\phi$ is injective. Note that the $\psi_i^j$ are zero if $i\neq 0$ and $i\neq j$. The morphisms $\psi_j^j$ are isomorphisms by the same reason as $\phi^0$ is. Thus, $\psi$ is surjective. Moreover, we see that a local section in the kernel of $\psi$ is determined by its component in $T^q_{[k+1],[k+1]}$. On the other hand, for every given local section $s$ of $T^q_{[k+1],[k+1]}$ there is a section in the image of $\phi$ whose component in $T^q_{[k+1],[k+1]}$ equals $s$ because of $\phi^0$ being an isomorphism. Since the rows in $B(k)_1$ are complexes, this already shows that $\im(\phi)=\ker(\psi)$.        
So by now we have seen that all $\ell$-th rows with $\ell> 0$ are exact on the 1-level what implies the vanishing of $B(k)^{p,\ell}_2$ for all $\ell>0$ and all $p$.
For $\ell=2k-q=0$, i.e. $q=2k$, the $(\sym_{k+1}\times\overline {\sym_{k+1}})$-invariants of the sequence (3) reduce by (2) and remark \ref{IJvanish} to the two term complex 
\begin{align*}0\to [\sHom(E_{[k+1]},F_{[k+1]})]^{\sym_{[k+1]}\times \sym_{[k+2,n]}}\overset\psi\to [\sHom(E_{[k+1]},F_{[k+2]})]^{\sym_{[k+1]}\times \sym_{[k+3,n]}}\to 0\tag 5\,.\end{align*}
The fact that the other terms vanish can be seen either directly by looking at the description of $T^{2k}_{I,J}$ or using the fact that for two sheaves 
$\E,\F$ on a variety which are push-forwards of locally free sheaves along closed immersions $\sHom(\E,\F)$ is non-trivial only if $\supp \E\supset \supp \F$.    
The first term of (5) is naturally isomorphic to $\sHom(E,F)_{[k+1]}^{\sym_{[k+1]}\times \sym_{[k+2,n]}}$ and the second one to $\sHom(E,F)_{[k+2]}^{\sym_{[k+1]}\times \sym_{[k+3,n]}}$. Under these identifications the morphism $\psi$ is just given by restricting local sections to $\Delta_{[k+2]}$. Thus, by lemma \ref{invsur} it is surjective which makes $B(k)_2^{0,k+1}$ vanish. 
The map $\psi$ is not injective because the support of its domain is larger than the support of its codomain. So indeed, $B(k)^{k,0}_2$ is the only non-trivial term on the 2-level.
\end{proof}
\begin{cor}
Let $E$, $F$ be locally free sheaves on $X$ and $k=0\dots,n-1$. Then the object $[R\sHom(C^k_E,p_*q^*F^{[n]})]^{\sym_n}$ is cohomologically concentrated in degree $k$, i.e. for $m\neq k$ \[ [\sExt^m(C^k_E,p_*q^*F^{[n]})]^{\sym_n}=0\,.\]  
\end{cor}
\begin{proof}
 The above lemma implies in particular that $B(k)$ degenerates at the 2-level. Thus, $B(k)^{p,q}_2=B(k)^{p,q}_\infty$ for all $p,q\in\Z$ and the only non-trivial term on the $\infty$-level is $B(k)^{k,0}_\infty$. Hence whenever $m\neq k$ we have 
\[0=B(k)^m=[\sExt^m(C^k_E,p_*q^*F^{[n]})]^{\sym_n}\,.\]
\end{proof}
\begin{prop}\label{RHom}
 Let $E$ and $F$ be locally free sheaves on $X$. Then $[R\sHom(\Phi(E),\Phi(F))]^{\sym_n}$ is cohomogically concentrated in degree zero and there is a natural isomorphism
\[[R\sHom(\Phi(E^{[n]}),\Phi(F^{[n]}))]^{\sym_n}\simeq [\sHom(C^0_E,C^0_F)]^{\sym_n}\,.\]
\end{prop}
\begin{proof}
As mentioned above we use the spectral sequence 
\[A_1^{p,q}=[\sExt^q(C^{-p}_E,p_*q^*F^{[n]})]^{\sym_n}\Longrightarrow A^m= [\sExt^m(C^\bullet_E,p_*q^*F^{[n]})]^{\sym_n}\simeq
[\sExt^m(\Phi(E^{[n]}),\Phi(F^{[n]})]^{\sym_n}\,.\]
By the above corollary the 1-sheet of $A$  is concentrated on the diagonal $p+q=0$. Thus $A^m=0$ for $m\ne 0$. This yields
\[[R\sHom(\Phi(E),\Phi(F))]^{\sym_n}\simeq [\sHom(p_*q^*(E^{[n]}),p_*q^*(F^{[n]}))]^{\sym_n}  \overset{\ref{invhom}}\simeq [\sHom(C^0_E,C^0_F)]^{\sym_n}\,.\]  
\end{proof}
\subsection{From tautological bundles to tautological objects}
\begin{prop}\label{RHomobj}
 Let $E^\bullet, F^\bullet\in\D^b(X)$. Then there are natural isomorphisms
\begin{align*}
 \left[R\sHom(\Phi((E^\bullet)^{[n]}),\Phi((F^\bullet)^{[n]})\right]^{\sym_n}&\simeq \left[R\sHom(C^0_{E^\bullet},C^0_{F\bullet}\right]^{\sym_n}\,,\\
\left[R\sHom(\Phi((E^\bullet)^{[n]}), \reg_{X^n})\right]^{\sym_n}&\simeq \left[R\sHom(C^0_{E^\bullet}, \reg_{X^n})\right]^{\sym_n}\,.\\
\end{align*}
\end{prop}
\begin{proof}
We take locally free resolutions $A^\bullet\simeq E^\bullet$ and $B^\bullet\simeq F^\bullet$ of the complexes $E^\bullet$ and $F^\bullet$.
Then $(A^\bullet)^{[n]}\simeq (E^\bullet)^{[n]}$ and also
\[\Phi((E^\bullet)^{[n]})\simeq \Phi((A^\bullet)^{[n]})\overset{\ref{Scadesc}}\simeq p_*q^*(A^{[n]})^\bullet\,.\]
Clearly, there are the same isomorphisms for $F^\bullet$ and $B^\bullet$ instead of $E^\bullet$ and $A^\bullet$.
Now, for every pair $i,j\in \Z$ by proposition \ref{RHom} we have
\[\iext^q(p_*q^*((A^i)^{[n]}),p_*q^*( (B^j)^{[n]})=0\]
for $q\neq 0$. Thus, we can apply proposition \ref{bifun} to the situation of the bifunctor
\[\ihom(\_,\_)=[\_]^{\sym_n}\circ\pi_*\circ \sHom(\_,\_)\colon \QCoh_{\sym_n}(X^n)^\circ \times \QCoh_{\sym_n}(X^n)\to \QCoh(S^nX)\]
and obtain using the natural isomorphism of \ref{RHom}, remark \ref{natural}, and the exactness of $C^0$
\begin{align*}
\left[R\sHom(\Phi((E^\bullet)^{[n]}),\Phi((F^\bullet)^{[n]})\right]^{\sym_n}
&\simeq R\ihom(p_*q^*((A^\bullet)^{[n]}),p_*q^*((B^\bullet)^{[n]}))\\
&\simeq \ihom^\bullet((p_*q^*((A^\bullet)^{[n]}),p_*q^*((B^\bullet)^{[n]})\\
&\simeq \left[\sHom^\bullet(C^0_{A^\bullet}, C^0_{B^\bullet})\right]^{\sym_n}\\
&\overset{\text{lf}}\simeq \left[R\sHom(C^0_{A^\bullet}, C^0_{B^\bullet})\right]^{\sym_n}\\
&\simeq \left[R\sHom(C^0_{E^\bullet}, C^0_{F^\bullet})\right]^{\sym_n}\,.
\end{align*}
The second isomorphism is shown similarly using proposition \ref{Rdual}.
\end{proof}
\subsection{Determinant line bundles}
There is a homomorphism which associates to any line bundle on $X$ its associated \textit{determinant line bundle} on $X^{[n]}$ given by  
\[\mathcal D\colon \Pic X\to \Pic X^{[n]}\quad ,\quad L\mapsto \mathcal D_L:= \mu^*((L^{\boxtimes n})^{\sym_n})\,.\]
Here the $\sym_n$-linearization of $L^{\boxtimes n}$ is given by the canonical isomorphisms $p_{\sigma^{-1}(i)}^*L\cong \sigma^*p_i^*L$, i.e. given by permutation of the factors.
By \cite[Theorem 2.3]{DN} the sheaf of invariants of $L^{\boxtimes n}$ is also the decent of $L^{\boxtimes n}$, i.e. $L^{\boxtimes n}\cong \pi^*((L^{\boxtimes n})^{\sym_n})$.
\begin{remark}
The functor $\mathcal D$ maps the trivial respectively the canonical line bundle to the trivial respectively the canonical line bundle, i.e. 
$\mathcal D_{\reg_X}\cong \reg_{X^{[n]}}$ and $\mathcal D_{\omega_X}\cong \omega_{X^{[n]}}$. The assertion for the trivial line bundle is true, since the pull-back of the trivial line bundle along any morphism is the trivial line bundle and since taking the invariants of the trivial line bundle yields the trivial line bundle on the quotient by the group action. For a proof of $\mathcal D_{\omega_X}\cong \omega_{X^{[n]}}$ see \cite[Proposition 1.6]{NW}.
\end{remark}
\begin{lemma}\label{det}
Let $L$ be a line bundle on $X$.
\begin{enumerate}
\item For every $\mathcal F^\bullet\in \D^b(X^{[n]})$ there is a natural isomorphism
$\Phi(\F^\bullet\otimes \mathcal D_L)\simeq \Phi(\F^\bullet)\otimes L^{\boxtimes n}$ in $\D_{\sym_n}(X^n)$. 
\item For every $\mathcal G^\bullet\in \D^b_{\sym_n}(X^n)$ there is in $\D^b(S^nX)$ a natural isomorphism \[[\pi_*(\mathcal G^\bullet\otimes L^{\boxtimes n})]^{\sym_n}\simeq
(\pi_*\mathcal G^\bullet)^{\sym_n}\otimes (\pi_*L^{\boxtimes n})^{\sym_n}\,.\]
\end{enumerate}
\end{lemma}
\begin{proof}
 By the definition of the determinant line bundle and the fact that $\pi^*(L^{\boxtimes n})^{\sym_n}\cong L^{\boxtimes n}$ we have 
\[q^*\mathcal D_L\cong q^*\mu^* (L^{\boxtimes n})^{\sym_n} \cong p^*\pi^*(L^{\boxtimes n})^{\sym_n}\cong p^*L^{\boxtimes n}\,.\]
Using this, we get indeed natural isomorphisms
\begin{align*}
\Phi(\F^\bullet\otimes \mathcal D_L)\simeq Rp_*q^*(\F^\bullet\otimes \mathcal D_L)\simeq Rp_*\left(q^*\F^\bullet\otimes q^*\mathcal D_L\right)&\simeq Rp_*\left(q^*\F^\bullet\otimes p^*L^{\boxtimes n} \right)\\
&\overset{\text{PF}}\simeq Rp_*q^*\F^\bullet\otimes L^{\boxtimes n}\\
&\simeq \Phi(\F^\bullet)\otimes L^{\boxtimes n}\,. 
\end{align*}
This shows (1). For (2) we remember that the functor $(\_)^{\sym_n}$ on $\D^b_{\sym_n}(X^n)$ is a abbreviation of the composition $(\_)^{\sym_n}\circ \pi_*$.
Then
\begin{align*}\left[\pi_*(\mathcal G^\bullet\otimes L^{\boxtimes n})\right]^{\sym_n}\simeq \left[\pi_*(\mathcal G^\bullet\otimes \pi^*(L^{\boxtimes n})^{\sym_n})\right]^{\sym_n}
 &\overset{\text{PF}}\simeq \left[\pi_*(\mathcal G^\bullet)\otimes (L^{\boxtimes n})^{\sym_n}\right]^{\sym_n}
 \\&\overset{\ref{tensinv}}\simeq (\pi_*\mathcal G^\bullet)^{\sym_n}\otimes (L^{\boxtimes n})^{\sym_n}\,.
\end{align*}
\end{proof}
We call an object of the form $(E^\bullet)^{[n]}\otimes \mathcal D_L\in \D^b(X^{[n]})$ with $E^\bullet \in \D^b(X)$ and $L$ a  line bundle on $X$ a \textit{twisted tautological object}.
\begin{prop}
 Let $(E^\bullet)^{[n]}\otimes \mathcal D_L$ and $(F^\bullet)^{[n]}\otimes \mathcal D_M$ be twisted tautological objects. Then there are natural isomorphisms
\begin{align*}
 \left[R\sHom(\Phi((E^\bullet)^{[n]}\otimes \mathcal D_L),\Phi((F^\bullet)^{[n]}\otimes \mathcal D_M))\right]^{\sym_n}&\simeq \left[R\sHom(C^0_{E^\bullet}\otimes L^{\boxtimes n},C^0_{F\bullet}\otimes M^{\boxtimes n})\right]^{\sym_n}\,,\\
\left[R\sHom((E^\bullet)^{[n]}\otimes \mathcal D_L, \mathcal D_M)\right]^{\sym_n}&\simeq \left[R\sHom(C^0_{E^\bullet}\otimes L^{\boxtimes n}, M^{\boxtimes n})\right]^{\sym_n}\,.\\
\end{align*}
\end{prop}
\begin{proof}
We will only show the first isomorphism since the proof of second one is very similar. We have indeed 
\begin{align*}
&\left[R\sHom(\Phi((E^\bullet)^{[n]}\otimes \mathcal D_L),\Phi((F^\bullet)^{[n]}\otimes \mathcal D_M))\right]^{\sym_n}\\
&\overset{\ref{det}}\simeq \left[R\sHom(\Phi((E^\bullet)^{[n]})\otimes L^{\boxtimes n},\Phi((F^\bullet)^{[n]})\otimes M^{\boxtimes n})\right]^{\sym_n}\\
&\overset{\text{lf}}\simeq \left[\left(R\sHom(\Phi((E^\bullet)^{[n]}),\Phi((F^\bullet)^{[n]}))\otimes (L^{\boxtimes n})^\vee\otimes M^{\boxtimes n}\right)\right]^{\sym_n}\\
&\overset{\ref{det}}\simeq [R\sHom(\Phi((E^\bullet)^{[n]}),\Phi((F^\bullet)^{[n]}))]^{\sym_n}\otimes ((L^{\boxtimes n})^\vee)^{\sym_n}\otimes (M^{\boxtimes n})^{\sym_n}\\
&\overset{\ref{RHomobj}}\simeq [\sHom(C^0_{E^\bullet},C^0_{F^\bullet})]^{\sym_n}\otimes ((L^{\boxtimes n})^\vee)^{\sym_n}\otimes (M^{\boxtimes n})^{\sym_n}\\
&\overset{\ref{det}}\simeq \left[\left(\sHom(C^0_{E^\bullet},C^0_{F^\bullet})\otimes (L^{\boxtimes n})^\vee\otimes M^{\boxtimes n}\right)\right]^{\sym_n}\\
&\overset{\text{lf}}\simeq [\sHom(C^0_{E^\bullet}\otimes L^{\boxtimes n},C^0_{F^\bullet}\otimes M^{\boxtimes n})]^{\sym_n}\,.
\end{align*}
\end{proof}
\subsection{Global Ext-groups}
\begin{theorem}\label{main}
Let $X$ be a smooth quasi-projective complex surface, $n\ge 2$, $E^\bullet,F^\bullet\in \D^b(X)$, and $L,M\in \Pic X$. The extension groups of the associated twisted tautological objects are given by the following natural isomorphisms of graded vector spaces:
\begin{align*}
\Ext^*((E^\bullet)^{[n]}\otimes \mathcal D_L,(F^\bullet)^{[n]}\otimes \mathcal D_M)&\cong \begin{aligned}& \Ext^*(E^\bullet\otimes L,F^\bullet\otimes M)\otimes S^{n-1}\Ext^*(L,M)\oplus\\ &\Ext^*(E^\bullet\otimes L,M)\otimes \Ext^*(L,F^\bullet\otimes M)\otimes S^{n-2}\Ext^*(L,M),\end{aligned}\\
 \Ext^*((E^\bullet)^{[n]}\otimes \mathcal D_L,\mathcal D_M)& \cong \Ext^*(E^\bullet\otimes L,M)\otimes S^{n-1}\Ext^*(L,M)\,.
\end{align*}
\end{theorem}    
\begin{proof}
Using the previous proposition and the considerations at the beginning of this section, the extension groups are given by
\begin{align*} &\Ext^*\left((E^\bullet)^{[n]}\otimes \mathcal D_L,(F^\bullet)^{[n]}\otimes \mathcal D_M\right)\\
\cong&\sym_n\Ext^*\left(\Phi((E^\bullet)^{[n]}\otimes \mathcal D_L),\Phi((F^\bullet)^{[n]}\otimes \mathcal D_M)\right)\\
\cong&\Ho^*(S^nX,\left[R\sHom(\Phi((E^\bullet)^{[n]}\otimes \mathcal D_L),\Phi((F^\bullet)^{[n]}\otimes \mathcal D_M))\right]^{\sym_n})\\
\cong& \Ho^*(S^nX,\left[R\sHom(C^0_{E^\bullet}\otimes L^{\boxtimes n},C^0_{F\bullet}\otimes M^{\boxtimes n})\right]^{\sym_n})\\
\cong& \left[\Ho^*(X^n,R\sHom(C^0_{E^\bullet}\otimes L^{\boxtimes n},C^0_{F\bullet}\otimes M^{\boxtimes n}))\right]^{\sym_n}\,.
\end{align*}
Applying the adjoint property of the inflation functor for $C^0_{E^\bullet}\otimes L^{\boxtimes}\simeq \Inf_{\sym_{\overline{[1]}}}^{\sym_n} (p_1^*E\otimes L^{\boxtimes n})$ and Danilas lemma we get 
\begin{align*}&\left[R\sHom(C^0_{E^\bullet}\otimes L^{\boxtimes n},C^0_{F^\bullet}\otimes M^{\boxtimes n})\right]^{\sym_n}\\
 \simeq& \left[\bigoplus_{j\in\n} R\sHom(p_1^*{E^\bullet}\otimes L^{\boxtimes n},p_j^*{F^\bullet}\otimes M^{\boxtimes n})\right]^{\overline{\sym_{[1]}}}\\
\simeq& \left[R\sHom(p_1^*{E^\bullet}\otimes L^{\boxtimes n},p_1^*{F^\bullet}\otimes M^{\boxtimes n})   \right]^{\overline{\sym_{[1]}}} \oplus
          \left[R\sHom(p_1^*{E^\bullet}\otimes L^{\boxtimes n},p_2^*{F^\bullet}\otimes M^{\boxtimes n})   \right]^{\overline{\sym_{[2]}}}\,.\tag{$\ast$}
\end{align*}
Using the compatibility of the derived sheaf-Hom with pullbacks gives 
\begin{align*}R\sHom(p_1^*{E^\bullet}\otimes L^{\boxtimes n},p_1^*{F^\bullet}\otimes M^{\boxtimes n})\simeq R\sHom({E^\bullet}\otimes L,{F^\bullet}\otimes M)\boxtimes \sHom( L,M)^{\boxtimes n-1}
\end{align*}
and  
\begin{align*}&R\sHom(p_1^*{E^\bullet}\otimes L^{\boxtimes n},p_2^*{F^\bullet}\otimes M^{\boxtimes n})\\ \simeq &R\sHom({E^\bullet}\otimes L,M)\boxtimes\sHom(L,F^\bullet\otimes M)\boxtimes \sHom( L,M)^{\boxtimes n-2}\,. 
\end{align*}
Now by the Künneth formula
\begin{align*}&\left[ \Ho^*\left(X^n,R\sHom({E^\bullet}\otimes L,{F^\bullet}\otimes M)\boxtimes \sHom( L,M)^{\boxtimes n-1}\right)\right]^{\overline{\sym_{[1]}}}\\
 \cong&\left[ \Ho^*(R\sHom({E^\bullet}\otimes L,{F^\bullet}\otimes M))\otimes \Ho^*(\sHom( L,M))^{\otimes n-1}\right]^{\overline{\sym_{[1]}}}\\
\cong&\left[ \Ext^*({E^\bullet}\otimes L,{F^\bullet}\otimes M))\otimes \Ext^*( L,M)^{\otimes n-1}\right]^{\overline{\sym_{[1]}}}\\
\cong& \Ext^*({E^\bullet}\otimes L,{F^\bullet}\otimes M))\otimes S^{n-1}(\Ext^*( L,M))\,.
\end{align*}
Doing the same for the other direct summand in $(\ast)$ yields the result. The proof of the second formula is again similar and therefore omitted.
\end{proof}
\begin{remark}
The above formulas are natural in $E^\bullet$, $F^\bullet$, $L$, and $M$, in automorphisms of $X$, and also in pull-backs along open immersions $U\subset X$.
\end{remark}
\begin{remark}
By setting $L=M=\reg_X$ we get the following formulas for non-twisted tautological objects  
\begin{align*}
\Ext^*((E^\bullet)^{[n]},(F^\bullet)^{[n]})&\cong \begin{aligned}& \Ext^*(E^\bullet,F^\bullet)\otimes S^{n-1}(\Ho^*(\reg_X))\oplus\\ &\Ho^*((E^\bullet)^\dv)\otimes \Ho^*(E^\bullet)\otimes S^{n-2}(\Ho^*(\reg_X))\end{aligned}\\
 \Ext^*((E^\bullet)^{[n]},\reg_X) \cong \Ho^*(X^{[n]},((E^\bullet)^{[n]})^\dv)&\cong \Ho^*((E^\bullet)^\dv)\otimes S^{n-1}\Ho^*(\reg_X)\,.
\end{align*}
\end{remark}
\begin{remark}\label{objcoh}
By the same arguments as used in this and the previous subsection we can also generalise Scala's formula for the cohomology of twisted tautological sheaves 
(see formula (1.1) of the  introduction) to twisted tautological objects. Namely, for every tautological object $(E^\bullet)^{[n]}\otimes \mathcal D_L$ there is a natural isomorphism of graded vector spaces
\[\Ho^*(X^{[n]}, (E^\bullet)^{[n]}\otimes \mathcal D_L)\cong \Ho^*(E^\bullet\otimes L)\otimes S^{n-1} \Ho^*(L)\,.\]  
Also, since $\mathcal D_L^\vee\otimes \mathcal D_M\cong \mathcal D_{\sHom(L,M)}$ for $L,M\in\Pic(X)$, we have a formula for
\[\Ext^*(\mathcal D_L,(E^\bullet)^{[n]}\otimes \mathcal D_M)\cong \Ho^*(X^{[n]}, (E^\bullet)^{[n]}\otimes \mathcal D_{\sHom(L,M)})\,.\]
\end{remark}
\begin{remark}
In the case that $X$ is projective one can also directly deduce the formula for $\Ext^*((E^\bullet)^{[n]}\otimes \mathcal D_L,\mathcal D_M)$ by Serre duality from Scala's formula in the form of the previous remark and the fact that $\mathcal D_{\omega_X}=\omega_{X^{[n]}}$.
\end{remark}
\begin{remark}
Using proposition \ref{invhom} we also get for $E_1,\dots, E_k, F$ locally free sheaves on $X$ formulas for 
$\Ext^0(E_1^{[n]}\otimes \dots \otimes E_k^{[n]},\reg_{X^{[n]}})$ as well as for $\Ext^0(E_1^{[n]}\otimes \dots \otimes E_k^{[n]},F^{[n]})$. For $\ell\in \N$ and $M\subset \N$ let $P(M,\ell)$ denote the set of partitions $I=\{I_1,\dots,I_\ell\}$ of the set $M$ of lenght $\ell$. Then 
\[\Hom(E_1^{[n]}\otimes \dots \otimes E_k^{[n]},\reg_{X^{[n]}})\cong \bigoplus_{I\in P([k],\ell),\ell\le n}\left(\bigotimes_{s=1}^\ell\Ho^0(\otimes_{i\in I_s} E_i^\vee)\otimes S^{n-\ell}\Ho^0(\reg_X)\right)\]
and $\Hom(E_1^{[n]}\otimes \dots \otimes E_k^{[n]},F^{[n]})$ is isomorphic to
\[\bigoplus_{\substack{M\subset[k]\\I\in P([k]\setminus M,\ell),\ell\le n-1}}\left(\Hom(\otimes_{i\in M} E_i,F)\otimes \bigotimes_{s=1}^\ell\Ho^*(\otimes_{i\in I_s} E_i^\vee)\otimes S^{n-\ell-1}\Ho^0(\reg_X)\right)\,.\]
Again, there are similar formulas for the sheaves twisted by determinant line bundles. However, these formulas can not directly be generalized to formulas for $\Ext^*$ since the corresponding $R\ihom$ objects are in general not cohomologically concentrated in degree zero for $k\ge 2$.
\end{remark}
\subsection{Spherical and $\mathbbm P^n$-objects}
Let $X$ be a smooth projective variety with canonical bundle $\omega_X$. We call an object $\Eb\in \D^b(X)$ an \textit{$\omega_X$-invariant} object if $\Eb\otimes \omega_X\simeq \Eb$. A \textit{spherical object} in $\D^b(X)$ is an $\omega_X$-invariant object $\Eb$ with the property
\[\Ext^i(\Eb,\Eb)=\begin{cases}
                   \C \quad&\text{if $i=0,\dim(X)$,}\\
                   0 \quad &\text{else,}
                  \end{cases}\]
i.e. $\Ext^*(\Eb,\Eb)\cong \Ho^*(S^{\dim X}, \C)$, where $S^n$ denotes the real $n$-sphere. A \textit{$\mathbbm P^n$-object} in the derived category $\D^b(X)$ is a $\omega_X$-invariant object $\Fb$ such that there is an isomorphism of graded algebras $\Ext^*(\Fb,\Fb)\cong H^*(\mathbbm P^n,\C)$, where  the multiplication on the left is given by the Yoneda product and on the right by the cup product. In particular for the underlying vector spaces
\[\Ext^i(\Eb,\Eb)=\begin{cases}
                   \C \quad&\text{if $0\le i\le 2n$ is even}\\
                   0 \quad &\text{if $i$ is odd}
                  \end{cases}\]
holds. By Serre duality the dimension of $X$ must be $2n$ as soon as $\D^b(X)$ contains a $\mathbbm P^n$-object. Spherical and $\mathbbm P^n$-objects are of interest because they induce automorphisms of $\D^b(X)$ (see \cite[chapter 8]{Huy}). For a smooth projective surface $X$ with $\omega_X=\reg_X$ the canonical bundle on $X^{[n]}$ is also trivial (see remark \ref{det}). Hence, the property of being $\omega_{X^{[n]}}$-invariant is automatically satisfied for every object in $\D^b(X^{[n]})$. Thus, one could hope that there are tautological objects induced by some special objects in $\D^b(X)$ that are spherical or $\mathbbm P^n$-objects. But this is not the  case by the  following proposition.
\begin{prop}
\begin{enumerate}
 \item Let $X$ be a smooth projective surface with trivial canonical line bundle and $n\ge 2$. Then twisted tautological objects on $X^{[n]}$ are never spherical or $\mathbbm P^n$-objects.
\item Let $X$ be a smooth projective surface and $n\ge 3$. Then twisted tautological sheaves on $X^{[n]}$ are never spherical or $\mathbbm P^n$-objects.
\end{enumerate}
\end{prop}
\begin{proof}
Let $(E^\bullet)^{[n]}\otimes\mathcal D_L$ be a non-zero twisted tautological object on $X^{[n]}$. Then using the fact that $\sHom(L,L)\cong \reg_X$ we have \[\Ext^*((E^\bullet)^{[n]}\otimes\mathcal D_L,(E^\bullet)^{[n]}\otimes\mathcal D_L)\cong \Ext^*((E^\bullet)^{[n]},(E^\bullet)^{[n]})\,.\]
Thus, we may assume $L=\reg_X$. The graded vector space $\Ext^*((E^\bullet)^{[n]},(E^\bullet)^{[n]})$ is given by theorem \ref{main} by
\[\Ext^*(E^\bullet,E^\bullet)\otimes S^{n-1}\Ho^*(\reg_X)\oplus \Ho^*((E^\bullet)^\dv)\otimes \Ho^*(E^\bullet)\otimes S^{n-2} \Ho^*(\reg_X)\,.\]
If $\omega_X=\reg_X$ we have by Serre duality $h^2(\reg_X)=h^0(\reg_X)=1$. The vector space $\Ext^0(E^\bullet,E^\bullet)$ has positive dimension since it contains the identity. Thus $\Ext^0(E^\bullet,E^\bullet)\otimes S^{n-1}\Ho^*(\reg_X)$ contributes non-trivially to the degrees $0,2,\dots,2n-2$ of $\Ext^*((E^\bullet)^{[n]},(E^\bullet)^{[n]})$. But by Serre duality also $\Ext^2(E^\bullet, E^\bullet)$ is non-vanishing. Thus also $\Ext^2(E^\bullet,E^\bullet)\otimes S^{n-1}\Ho^*(\reg_X)$ contributes non-trivially to the degrees $2,4,\dots,2n$. This yields $\ext^i((E^\bullet)^{[n]},(E^\bullet)^{[n]})\ge 2$ for $i=2,4,2n-2$ which shows that $(E^\bullet)^{[n]}\otimes\mathcal D_L$ is indeed neither a spherical nor a $\mathbbm P^n$-object.   
Now let $n\ge 3$ and $E^{[n]}$ be a tautological sheaf. Extension groups of sheaves on $X$ can be non-trivial only in the degrees $0,1,2$. If $E^{[n]}$ was a spherical or a $\mathbbm P^n$-object then in particular the highest and the lowest extension groups, i.e. in dergee $0$ and $2n$, must not vanish. Since for $n\ge 3$ the term $\Ho^*(\reg_X)$ occurs in both direct summands of $\Ext^*((E^\bullet)^{[n]},(E^\bullet)^{[n]})$ it follows that
$\Ho^i(\reg_X)\neq 0$ for $i=0,2$. Furthermore, either $\Ext^2(E,E)$ or $(\Ext^*(E,\reg_X)\otimes \Ext^*(\reg_X,E))^4$ must not vanish. In both cases $\ext^4(E^{[n]},E^{[n]})\ge 2$ since also $\Ext^0(E,E)\otimes S^{n-1}\Ho^*(\reg_X)$ contributes non-trivially to $\Ext^4(E^{[n]},E^{[n]})$.   
\end{proof}

\section{Yoneda products and interpretation of the results}
\subsection{Yoneda products, the Künneth isomorphism and signs}\label{yonsign}
Let $\mathcal A$ be an abelian category with enough injectives. Recall that for every $A^\bullet\in \D^-(\mathcal A)$ and $B^\bullet \in\D^+(\mathcal A)$ there are natural isomorphisms $\Ext^i(A^\bullet,B^\bullet)\cong \Hom_{\D(\mathcal A)}(A^\bullet, B^\bullet[i])$. For $A^\bullet,B^\bullet, C^\bullet\in \D^b(\mathcal A)$ the \textit{Yoneda product} is the bilinear map $\Yon$ making the following diagram commute 
\[\xymatrix{
 \Ext^j(B^\bullet,C^\bullet) \ar^\cong[d]& \times & \Ext^i(A^\bullet,B^\bullet)\ar^\cong[d]\ar^\Yon[r] & \Ext^{i+j}(A^\bullet,C^\bullet)\ar^\cong[d]\\
\Hom_{\D(\mathcal A)}(B^\bullet[i], C^\bullet[i+j]) & \times & \Hom_{\D(\mathcal A)}(A^\bullet, B^\bullet[i])\ar^\circ[r] & \Hom_{\D(\mathcal A)}(A^\bullet, C^\bullet[i+j])\,.
}
\]
The second and third vertical arrows are the isomorphisms mentioned and the first is the mentioned isomorphism composed with the $i$-th  power of the shift functor of $\D(\mathcal A)$. The lower horizontal bilinear map is just the composition law in $\D(\mathcal A)$. In the following we will treat the vertical isomorphisms as they were identities and consequently denote the Yoneda product just by $\circ$. Considering the above products as maps on the homogeneous components we get the Yoneda product as a bilinear map of graded vector spaces
\[\Ext^*(B^\bullet,C^\bullet)  \times  \Ext^*(A^\bullet,B^\bullet)\overset\circ \to \Ext^*(A^\bullet,C^\bullet)\,.\]
There will occur signs because of the use of the Künneth isomorphism. Let $X,Y$ be varieties together with objects $A_1^\bullet,B_1^\bullet,C_1^\bullet\in \D^b(X)$ and $A_2^\bullet,B_2^\bullet,C_2^\bullet\in \D^b(Y)$. We consider homogeneous elements
$a_i\in \Ext^*(A_i^\bullet,B_i^\bullet)$ and  $b_i\in \Ext^*(B_i^\bullet,C_i^\bullet)$ for $i=1,2$. Then via the Künneth isomorphism we can interpret the tensor products as elements of the extension groups on the product $X\times Y$ namely
\[a_1\otimes a_2\in \Ext^*(A_1^\bullet\boxtimes A_2^\bullet,B_1^\bullet\boxtimes B_2^\bullet)\,,\, b_1\otimes b_2\in \Ext^*(B_1^\bullet\boxtimes B_2^\bullet,C_1^\bullet\boxtimes C_2^\bullet)\,.\]
The Yoneda product is then given by 
\[(b_1\otimes b_2)\circ (a_1\otimes a_2)=(-1)^{\deg b_2\deg a_1}(b_1a_1)\otimes (b_2a_2)\,,\]
where we omit the $\circ$ for the Yoneda products on $X$ and $Y$. The occurrence of the sign can be seen best when defining the Yoneda product using the cup product (see e.g \cite[section 10.1.1]{HL}).   
To capture these signs we use the following convention: Let $a_1,\dots, a_n$ and $b_1=a_{n+1},\dots,b_n=a_{2n}$ be elements in any graded algebraic objects of degree $\deg(a_i)=p_i$. Let $T$ be a term in which all the $a_i$ and $b_i$ occur. Let 
$\sigma\in\sym_{2n}$ be the permutation such that after erasing everything in $T$ besides the $a_i$ and $b_i$ we get $a_{\sigma^{-1}(1)}\dots a_{\sigma^{-1}(2n)}$. Then we set $\eps(T):=\eps_{\sigma,p_1,\dots,p_{2n}}$ (see section \ref{not}). Let $I$ be a finite set and for each $i\in I$ let $T_i$ be a term in which all the $a_i$ and $b_i$ occur. We define
\[\sum_{i\in I}^\bullet T_i:= \sum_{i\in I} \eps(T_i)\cdot T_i\,.\]
In the following we will have such sums where $a$ and $b$ will be replaced by two other letters. Then always the first letter from the left in each summand will be the same. This letter is considered as $a$ and the other as $b$.
For example if we have $x_1,x_2,y_1,y_2$ all of odd degree then
\begin{align*}
&\sum_{\sigma,\tau\in\sym_2}^\bullet (x_{\sigma^{-1}(1)}\circ y_{\tau^{-1}(1)})\otimes (x_{\sigma^{-1}(2)}\circ y_{\tau^{-1}(2)})\\  
=&-(x_1\circ y_1)\otimes(x_2\circ y_2)+ (x_1\circ y_2)\otimes(x_2\circ y_1)+ (x_2\circ y_1)\otimes(x_1\circ y_2)- (x_2\circ y_2)\otimes(x_1\circ y_1)\,.
\end{align*}
\subsection{Yoneda products for twisted tautological objects}
Let $E^\bullet,F^\bullet\in \D^b(X)$ and $L,M$ be line bundles on $X$. In the last section we computed formulas for $\Ext^*((E^\bullet)^{[n]}\otimes \mathcal D_L,\mathcal D_M)$ as well as for 
$\Ext^*((E^\bullet)^{[n]}\otimes\mathcal D_L,(F^\bullet)^{[n]}\otimes \mathcal D_M)$. It was done using natural isomorphisms
\begin{align*}
   \Ext^*((E^{\bullet})^{[n]}\otimes \mathcal D_L,\mathcal D_M)\cong \sym_n\Ext^*(\Phi((E^\bullet)^{[n]}\otimes \mathcal D_L),\Phi(\mathcal D_M))\cong \sym_n\Ext^*(C^0_{E^\bullet}\otimes L^{\boxtimes n},M^{\boxtimes n})\tag1
\end{align*}
respectively
\begin{align*}
\Ext^*((E^\bullet)^{[n]}\otimes \mathcal D_L,(F^\bullet)^{[n]}\otimes \mathcal D_M)&\cong \sym_n\Ext^*(\Phi((E^\bullet)^{[n]}\otimes \mathcal D_L),\Phi((F^\bullet)^{[n]}\otimes \mathcal D_M))\tag2\\&\cong \sym_n\Ext^*(C^0_{E^\bullet}\otimes L^{\boxtimes n},C^0_{F^\bullet}\otimes M^{\boxtimes n})\,.  
\end{align*}
There is also a formula for $\Ext^*(\mathcal D_L,(F^\bullet)^{[n]}\otimes \mathcal D_M)$ (see remark \ref{objcoh}) using the isomorphisms
\begin{align*}
 \Ext^*(\mathcal D_L,(F^\bullet)^{[n]}\otimes \mathcal D_M)\cong \sym_n\Ext^*(\Phi(\mathcal D_L),\Phi((F^\bullet)^{[n]}\otimes \mathcal D_M))\cong \sym_n\Ext^*(L^{\boxtimes n},C^0_{F^\bullet}\otimes M^{\boxtimes n})\,.\tag3
\end{align*}
Furthermore, the formula 
\begin{align*}\Ext^*(\mathcal D_L,\mathcal D_M)\simeq S^n \Ext^*(L,M)\tag 4\end{align*}
can easily be proven using the isomorphisms (see lemma \ref{det})
\begin{align*}
  \Ext^*(\mathcal D_L,\mathcal D_M)\cong{\sym_n}\Ext^*(\Phi(\mathcal D_L),\Phi(\mathcal D_M))\cong{\sym_n}\Ext^*(L^{\boxtimes n},M^{\boxtimes n})\,.\tag5  
\end{align*}
In summary, we have formulas for the extension groups $\Ext^*(\Eb,\Fb)$ on $X^{[n]}$ whenever each of $\Eb$ and $\Fb$ is a twisted tautological object or a determinant line bundle. Clearly the Yoneda products 
\[\Ext^*(\Fb,\Gb)\times \Ext^*(\Eb,\Fb)\to \Ext^{*}(\Eb,\Gb)\]
coincide under the  Bridgeland-King-Reid equivalence with the Yoneda products 
\[\sym_n\Ext^*(\Phi(\Fb),\Phi(\Gb))\times \sym_n\Ext^*(\Phi(\Eb),\Phi(\Fb))\to \sym_n\Ext^{*}(\Phi(\Eb),\Phi(\Gb))\]
because $\Phi$ is a functor. Now for locally free sheaves $A$ and $B$ on $X$ the isomorphism 
\[[\sHom(C^0_A,C^0_B)]^{\sym_n}\overset \cong \to[\sHom(p_*q^*(A^{[n]}),p_*q^*(B^{[n]}))]^{\sym_n}\]
is given by restricting the $\sym_n$-equivariant morphisms to the subsheaf $p_*q^*(A^{[n]})$ of $C^0_A$  and observing that the restricted morphisms factorise through the subsheaf $p_*q^*(B^{[n]})$ of $C^0_B$ (see proposition \ref{invhom}). Also the isomorphisms $[\sHom(C^0_A,\reg_{X^n})]^{\sym_n}\overset \cong \to[\sHom(p_*q^*(A^{[n]}),\reg_{X^n})]^{\sym_n}$ and $[\sHom(\reg_{X^n},C^0_B)]^{\sym_n}\overset \cong \to[\sHom(\reg_{X^n},p_*q^*(B^{[n]}))]^{\sym_n}$ are given by restriction of the morphisms to subsheaves. 
Now the isomorphisms on the right hand sides of (1), (2) and (3) are induced by those isomorphisms of the sheaf-Homs after choosing locally free resolutions of $E^\bullet$ and $F^\bullet$. The composition of morphisms of sheaves is compatible with restricting the morphisms to certain subsheaves. This translates into the Yoneda products between the extension groups of twisted tautological objects and determinant line bundles coinciding with the  Yoneda products between the equivariant extension groups of the terms of the form $C^0_{E^\bullet}$, $L^{\boxtimes n}$, and $C^0_{E^\bullet}\otimes L^{\boxtimes n}$ under the isomorphisms in (1), (2), (3) and (5). Now using the Künneth and Danila's isomorphism we can express the Yoneda products in terms of Yoneda products in the derived category of the surface $X$ under the isomorphisms of theorem \ref{main}, remark \ref{objcoh} and (4) . We will explicitly  state and prove the formula only for the case were all the  objects $\Eb$, $\Fb$ and $\Gb$ involved are twisted tautological objects. 
The other seven cases can be done very similarly. For $E,F\in \D^b(X)$ and $L,M\in \Pic(X)$ we set
\[P(E,L,F,M):=\begin{aligned}&\Ext^*(E\otimes L,F\otimes M)\otimes S^{n-1}(\Ext^*(L,M))\oplus\\& \Ext^*(E\otimes L,M)\otimes \Ext^*(L,F\otimes M)\otimes S^{n-2}(\Ext^*(L,M))\,.\end{aligned} 
\]
Let also be $G\in \D^b(X)$ and $N\in\Pic (X)$. We consider the elements
\begin{align*}
 \binom{\phi\otimes s_2\cdots s_n}{\eta\otimes x\otimes t_3\cdots t_n}&\in P(F,M,G,N)\cong \Ext^*(F^{[n]}\otimes \mathcal D_M, G^{[n]}\otimes \mathcal D_N)\\
\binom{\phi'\otimes s'_2\cdots s'_n}{\eta'\otimes x'\otimes t'_3\cdots t'_n}&\in P(E,L,F,M)\cong \Ext^*(E^{[n]}\otimes \mathcal D_L, F^{[n]}\otimes \mathcal D_M)\,.
\end{align*}
In order to use the sign convention above we set
\[\phi=s_1\,,\, \eta=t_1\,,\, x=t_2\,,\, \phi'=s'_1\,,\, \eta'=t'_1\,,\, x'=t'_2\,.\]      
and assume that all the $s_i$, $s'_i$, $t_i$ and $t'_i$ are  homogeneous. Now we can compute the Yoneda product
$\binom{\phi\otimes s_2\cdots s_n}{\eta\otimes x\otimes t_3\cdots t_n}\circ\binom{\phi'\otimes s'_2\cdots s'_n}{\eta'\otimes x'\otimes t'_3\cdots t'_n}$ in 
$\Ext^*(E^{[n]}\otimes \mathcal D_L, G^{[n]}\otimes \mathcal D_N)$ and express it as an element in $P(E,L,G,M)$.
\begin{prop}
The Yoneda product is given by
\begin{align*}
 &\frac1{(n-1)!}\sum_{\sigma\in \sym_{[2,n]}}^\bullet (\phi\phi')\otimes(s_2 s'_{\sigma^{-1}(2)})\cdots (s_n s'_{\sigma^{-1}(n)})\\
&+\frac1{(n-2)!}\sum_{\tau\in \sym_{[3,n]}}^\bullet (x\eta')\otimes(\eta x')\cdot(t_3 t'_{\tau^{-1}(3)})\cdots (t_n t'_{\tau^{-1}(n)})\\
&\oplus\frac1{(n-1)!}\sum_{\beta\in \sym_{[2,n]}}^\bullet (\eta\phi')\otimes(x s'_{\beta^{-1}(2)})\otimes(t_3 s'_{\beta^{-1}(3)})\cdots (t_n s'_{\beta^{-1}(n)})\\
&+\frac1{(n-1)!}\sum_{\gamma\in \sym_{[2,n]}}^\bullet (s_{\gamma^{-1}(2)}\eta')\otimes(\phi x')\otimes(s_{\gamma^{-1}(3)}t'_3)\cdots (s_{\tau^{-1}(n)}t'_n)\\
&+\frac1{(n-2)!}\sum_{\substack{i=3,\dots,n\\ \alpha\in \sym_{[3,\dots,n]}}}^\bullet (t_i\eta')\otimes(x t'_{\alpha{-1}(i)})\otimes(\eta x')\cdot (t_3 t'_{\alpha^{-1}(3)})\cdots \widehat{(t_i t'_{\alpha^{-1}(i)})}\cdots(t_n t'_{\alpha^{-1}(n)})\,.
\end{align*}
\end{prop}
\begin{proof}
The element
$\binom{\phi\otimes s_2\cdots s_n}{\eta\otimes x\otimes t_3\cdots t_n}\in P(F,M,G,N)$
corresponds to the element
\[\frac1{(n-1)!}{\sum_{\sigma\in\sym_n}^\bullet s_{\sigma^{-1}(1)}\otimes \cdots\otimes s_{\sigma^{-1}(n)}}\oplus
\frac1{(n-2)!}
{\sum_{\tau\in\sym_n}^\bullet t_{\tau^{-1}(1)}\otimes \cdots\otimes t_{\tau^{-1}(n)}} 
\]
in $\Ext^*(C^0_F\otimes M^{\boxtimes n},C^0_G\otimes N^{\boxtimes n})$. The coefficients are coming from the canonical isomorphism $S^kV\overset\cong\to S_kV$ (see section \ref{not}, note that the isomorphism of Danilas lemma does not involve such coefficients). The same holds for $\binom{\phi'\otimes s'_2\cdots s'_n}{\eta'\otimes x'\otimes t'_3\cdots t'_n}$. The summand $s_{\sigma^{-1}(1)}\otimes \cdots\otimes s_{\sigma^{-1}(n)}$ is an element of $\Ext^*(\pr_{\sigma(1)}^*F\otimes M^{\boxtimes n},\pr_{\sigma(1)}^*G\otimes N^{\boxtimes n})$ and $t_{\tau^{-1}(1)}\otimes \cdots\otimes t_{\tau^{-1}(2)}$ an element of
$\Ext^*(\pr_{\tau(1)}^*F\otimes M^{\boxtimes n},\pr_{\tau(2)}^*G\otimes N^{\boxtimes n})$. There are five types of composable pairs of the components of the classes in $\Ext^*(C^0_F\otimes M^{\boxtimes n},C^0_G\otimes N^{\boxtimes n})\times
\Ext^*(C^0_E\otimes L^{\boxtimes n},C^0_F\otimes M^{\boxtimes n})\,,$ namely
\begin{align*}
&\pr_i^*E\otimes L^{\boxtimes n}\to \pr_i^*F\otimes M^{\boxtimes n}\to\pr_i^*G\otimes N^{\boxtimes n}\,,\,\pr_i^*E\otimes L^{\boxtimes n}\to \pr_j^*F\otimes M^{\boxtimes n}\to\pr_i^*G\otimes N^{\boxtimes n}\,,\\
&\pr_i^*E\otimes L^{\boxtimes n}\to \pr_i^*F\otimes M^{\boxtimes n}\to\pr_j^*G\otimes N^{\boxtimes n}\,,\,\pr_i^*E\otimes L^{\boxtimes n}\to \pr_j^*F\otimes M^{\boxtimes n}\to\pr_j^*G\otimes N^{\boxtimes n}\,,\\ &\pr_i^*E\otimes L^{\boxtimes n}\to \pr_j^*F\otimes M^{\boxtimes n}\to\pr_k^*G\otimes N^{\boxtimes n}  
\end{align*}
with $i,j,k\in [n]$ pairwise distinct. Thus, the Yoneda product in $\Ext^*(C^0_E\otimes L^{\boxtimes n},C^0_G\otimes N^{\boxtimes n})$ looks like this :
\begin{align*}
&\frac 1{(n-1)!^2}\sum_{\substack{\sigma,\sigma'\in\sym_n\\\sigma(1)=\sigma'(1)}}^\bullet \otimes_{i=1}^n(s_{\sigma^{-1}(i)} s'_{\sigma'^{-1}(i)})
+\frac 1{(n-2)!^2}\sum_{\substack{\tau,\tau'\in\sym_n\\ \tau(2)=\tau'(1),\tau(1)=\tau'(2)}}^\bullet \otimes_{i=1}^n(t_{\tau^{-1}(i)} t'_{\tau'^{-1}(i)})\\
&+\frac 1{(n-1)!(n-2)!}\sum_{\substack{\tau,\sigma'\in\sym_n\\\tau(1)=\sigma'(1)}}^\bullet \otimes_{i=1}^n(t_{\tau^{-1}(i)} s'_{\sigma'^{-1}(i)})
+\frac 1{(n-1)!(n-2)!}\sum_{\substack{\sigma,\tau'\in\sym_n\\\sigma(1)=\tau'(2)}}^\bullet \otimes_{i=1}^n(s_{\sigma^{-1}(i)} t'_{\tau'^{-1}(i)})\\
&+\frac 1{(n-2)!^2}\sum_{\substack{\tau,\tau'\in\sym_n\\\tau(1)=\tau'(2), \tau(2)\ne \tau'(1)}}^\bullet \otimes_{i=1}^n(t_{\tau^{-1}(i)} t'_{\tau'^{-1}(i)})\quad\tag 6\,.
\end{align*}
The first term of (6) is a $\sym_n$-invariant element of $\oplus_{i=1}^n\Ext^*(p_i^*E\otimes L^{\boxtimes n},p_i^*G\otimes N^{\boxtimes n})$. Danila's isomorphism is simply the projection to the first summand. Thus, it maps the first term of (6) to
\[\frac 1{(n-1)!^2}\sum_{\substack{\sigma,\sigma'\in\sym_n\\\sigma(1)=\sigma'(1)=1}}^\bullet \otimes_{i=1}^n(s_{\sigma^{-1}(i)} s'_{\sigma'^{-1}(i)})\in
\Ext^*(p_1^*E\otimes L^{\boxtimes n},p_1^*G\otimes N^{\boxtimes n})\]
Under the isomorphism $S_{n-1}\Ext^*(L,N)\overset\cong\to S^{n-1}\Ext^*(L,N)$ this element is mapped to 
\[\frac1{(n-1)!}\sum_{\sigma\in\sym_{[2,n]}}^\bullet (\phi\phi')\otimes (s_2 s'_{\sigma^{-1}(2)})\cdots (s_n s'_{\sigma^{-1}(n)})\in
 \Ext^*(E\otimes L,G\otimes N)\otimes S^{n-1}\Ext^*(L,N)\]
which is exactly the first term of the formula we want to prove.
Doing the same for the other four terms in (6) yields the desired element in
$P(E,L,G,M)$.   
\end{proof}
\begin{remark}
Let $D,E,F,G\in \D^b(X)$ and $L,M,N\in \Pic X$. By remark \ref{natural} for a morphism $\phi\in \Hom_{\D^b(X)}(E,F[i])\cong \Ext^i(E,F)$ the induced morphism \[\phi^{[n]}\otimes \id_{\mathcal D_M}\in \Ext^i(E^{[n]}\otimes \mathcal D_M,F^{[n]}\otimes \mathcal D_M)\] corresponds to 
\[\binom{\phi\otimes \id_{M}\cdots\id_M}0 \in P(E,M,F,M)\,.\]
For $M=\reg_X$ this shows that the tautologizing functor $(\_)^{[n]}$ is faithful but not full. Let 
\[\binom{\psi'\otimes s_2'\cdots s_n'}{\eta'\otimes x'\otimes t_3'\cdots t_n'}\in P(D,L,E,M)\,,\,
 \binom{\psi\otimes s_2\cdots s_n}{\eta\otimes x\otimes t_3\cdots t_n}\in P(F,M,G,N)\,.
\]
Then the formula for the Yoneda product gives back the naturalness of the isomorphism in theorem \ref{main} as a special case since
\begin{align*}
  \binom{\psi\otimes s_2\cdots s_n}{\eta\otimes x\otimes t_3\cdots t_n}\circ\binom{\phi\otimes \id_{M}\cdots\id_M}0&=\binom{\psi\phi\otimes s_2\cdots s_n}{\eta\phi\otimes x\otimes t_3\cdots t_n}\\
\binom{\phi\otimes \id_{M}\cdots\id_M}0 \circ\binom{\psi'\otimes s_2'\cdots s_n'}{\eta'\otimes x'\otimes t_3'\cdots t_n'}&=\binom{\phi\psi'\otimes s_2'\cdots s_n'}{\eta'\otimes \phi x'\otimes t_3'\cdots t_n'}\,.
\end{align*}
\end{remark}

\bibliographystyle{alpha}
\addcontentsline{toc}{chapter}{References}
\bibliography{references}%Name Deiner bib-Datei

\end{document}